\documentclass[12pt,a4paper]{amsart}
\usepackage{mathtools}
\usepackage{xcolor}
\usepackage{amsmath,amstext,amssymb,amsthm,amsfonts}
\usepackage{graphicx}
\usepackage[all]{xy} \xyoption{poly}
\usepackage{ mathrsfs }

\newcommand{\nc}{\newcommand}

\nc{\one}{\mbox{\bf 1}}
\nc{\invtensor}{\underset{\leftarrow}{\otimes}}
\nc{\const}{\operatorname{const}}

\nc{\re}{\operatorname{re}}
\nc{\ad}{\operatorname{ad}}
\nc{\ev}{\operatorname{ev}}
\nc{\tr}{\operatorname{tr}}

\nc{\Gr}{\mathscr{K}}
\nc{\rGr}{\operatorname{rGr}}
\nc{\atyp}{\operatorname{atyp}}
\nc{\tp}{\operatorname{top}}
\nc{\rank}{\operatorname{rank}}
\nc{\corank}{\operatorname{corank}}
\nc{\codim}{\operatorname{codim}}
\nc{\sdim}{\operatorname{sdim}}
\nc{\mult}{\operatorname{mult}}
\nc{\ds}{\operatorname{ds}}
\nc{\tail}{\operatorname{tail}}
\nc{\howl}{\operatorname{howl}}
\nc{\triv}{\operatorname{triv}}
\nc{\spn}{\operatorname{span}}
\nc{\Sym}{\operatorname{Sym}}
\nc{\Core}{\operatorname{Core}}
\nc{\id}{\operatorname{id}}
\nc{\Id}{\operatorname{Id}}
\nc{\Ree}{\operatorname{Re}}
\nc{\hi}{\operatorname{hi}}
\nc{\htt}{\operatorname{ht}}
\nc{\at}{\operatorname{at}}
\nc{\str}{\operatorname{str}}
\nc{\Iso}{\operatorname{Iso}}
\nc{\Ker}{\operatorname{Ker}}
\nc{\rker}{\operatorname{rKer}}
\nc{\im}{\operatorname{Im}}
\nc{\osp}{\mathfrak{osp}}
\nc{\sgn}{\operatorname{sgn}}
\nc{\F}{\operatorname{F}}
\nc{\Mod}{\operatorname{Mod}}
\nc{\DS}{\operatorname{DS}}
\nc{\Soc}{\operatorname{Soc}}
\nc{\Inj}{\operatorname{Inj}}
\nc{\Hom}{\operatorname{Hom}}
\nc{\End}{\operatorname{End}}
\nc{\supp}{\operatorname{supp}}
\nc{\smult}{\operatorname{smult}}
\nc{\Dyn}{\operatorname{Dyn}}
\nc{\Card}{\operatorname{Card}}
\nc{\Ann}{\operatorname{Ann}}
\nc{\Arc}{\operatorname{Arc}}
\nc{\arc}{\operatorname{arc}}
\nc{\Ind}{\operatorname{Ind}}
\nc{\Coind}{\operatorname{Coind}}
\nc{\wt}{\operatorname{hwt}}
\nc{\hwt}{\operatorname{wt}}
\nc{\rk}{\operatorname{rank}}
\nc{\ch}{\operatorname{ch}}
\nc{\sch}{\operatorname{sch}}
\nc{\mdim}{\operatorname{mdim}}
\nc{\Stab}{\operatorname{Stab}}
\nc{\hull}{\operatorname{hull}}

\nc{\Irr}{\operatorname{Irr}}
\nc{\Spec}{\operatorname{Spec}}
\nc{\Res}{\operatorname{Res}}
\nc{\res}{\operatorname{res}}
\nc{\Aut}{\operatorname{Aut}}
\nc{\Ext}{\operatorname{Ext}}
\nc{\Prec}{\operatorname{Prec}}
\nc{\Fract}{\operatorname{Fract}}
\nc{\gr}{\operatorname{gr}}
\nc{\vol}{\operatorname{vol}}
\nc{\diag}{\operatorname{diag}}
\nc{\deff}{\operatorname{def}}
\nc{\core}{\operatorname{core}}
\nc{\HC}{\operatorname{HC}}
\nc{\Ch}{\operatorname{Ch}}
\nc{\Sch}{\operatorname{Sch}}
\nc{\dpth}{\operatorname{dpth}}
\nc{\Pol}{\operatorname{Pol}}

\nc{\pari}{\operatorname{par}}
\nc{\pos}{\operatorname{pos}}

\nc{\Howl}{\operatorname{Subq}}

\nc{\Cl}{\mathcal{C}\ell}

\nc{\wdchi}{\widetilde{\chi}}
\nc{\wdH}{\widetilde{H}}
\nc{\wdN}{\widetilde{N}}
\nc{\wdM}{\widetilde{M}}
\nc{\wdO}{\widetilde{O}}
\nc{\wdR}{\widetilde{R}}

\nc{\wdV}{\widetilde{V}}

\nc{\wdC}{\widetilde{C}}

\nc{\zero}{\operatorname{zero}}
\nc{\nonzero}{\operatorname{nonzero}}

\nc{\Nonzero}{\operatorname{Nonzero}}

\nc{\diam}{\operatorname{diam}}

\nc{\Obj}{\operatorname{Obj}}
\nc{\Dglie}{\operatorname{{\mathcal D}glie}}
\nc{\Fin}{\operatorname{{\mathcal F}in}}
\nc{\pr}{\operatorname{pr}}
\nc{\Adm}{\operatorname{\mathcal{A}dm}}

\nc{\fg}{\mathfrak g}

\nc{\Sg}{{\cS(\fg)}}
\nc{\Shg}{{\cS(\fhg)}}
\nc{\Ug}{{\cU(\fg)}}
\nc{\Uhg}{{\cU(\fhg)}}
\nc{\Sh}{{\cS(\fh)}}
\nc{\Uh}{{\cU(\fh)}}
\nc{\Uhh}{{\cU(\fhh)}}
\nc{\Zg}{{{\mathcal{Z}}(\fg)}}

\nc{\Vir}{{\mathcal{V}ir}}
\nc{\NS}{{\mathcal{N}S}}

\nc{\tZg}{{\widetilde{\mathcal Z}({\mathfrak g})}}
\nc{\Zk}{{\mathcal Z}({\mathfrak k})}

\nc{\Up}{{\mathcal U}({\mathfrak p})}
\nc{\Ah}{{\mathcal A}({\mathfrak h})}
\nc{\Ag}{{\mathcal A}({\mathfrak g})}
\nc{\Ap}{{\mathcal A}({\mathfrak p})}
\nc{\Zp}{{\mathcal Z}({\mathfrak p})}
\nc{\cR}{\mathcal R}
\nc{\cS}{\mathcal S}
\nc{\cP}{\mathcal P}
\nc{\cT}{\mathcal{T}}
\nc{\cC}{\mathcal C}
\nc{\cJ}{\mathcal J}
\nc{\cA}{\mathcal A}
\nc{\cV}{\mathcal V}
\nc{\cU}{\mathcal U}
\nc{\cZ}{\mathcal Z}
\nc{\cM}{\mathcal M}
\nc{\cL}{\mathcal L}
\nc{\cF}{\mathcal F}

\nc{\cB}{\mathcal{B}}

\nc{\fo}{\mathfrak o}

\nc{\fa}{\mathfrak a}

\nc{\fz}{\mathfrak z}
\nc{\CO}{\mathcal O}
\nc{\CR}{\mathcal R}

\nc{\cK}{\mathcal{K}}
\nc{\cW}{\mathcal{W}}
\nc{\bM}{\mathbf{M}}
\nc{\bL}{\mathbf{L}}
\nc{\bN}{\mathbf{N}}

\nc{\zq}{\mathpzc q}

\nc{\fl}{\mathfrak l}
\nc{\fn}{\mathfrak n}
\nc{\fm}{\mathfrak m}
\nc{\fp}{\mathfrak p}
\nc{\fh}{\mathfrak h}
\nc{\ft}{\mathfrak t}
\nc{\fk}{\mathfrak k}
\nc{\fb}{\mathfrak b}
\nc{\fs}{\mathfrak s}
\nc{\fB}{\mathfrak B}

\nc{\vareps}{\varepsilon}
\nc{\varesp}{\varepsilon}
\nc{\veps}{\varepsilon}

\nc{\fsl}{\mathfrak{sl}}
\nc{\fgl}{\mathfrak{gl}}
\nc{\fso}{\mathfrak{so}}
\nc{\fosp}{\mathfrak{osp}}
\nc{\fsp}{\mathfrak{sp}}
\nc{\fq}{\mathfrak q}
\nc{\fsq}{\mathfrak{sq}}
\nc{\fpsq}{\mathfrak{psq}}
\nc{\fpq}{\mathfrak{pq}}

\nc{\fhg}{\hat{\fg}}
\nc{\fhn}{\hat{\fn}}
\nc{\fhh}{\hat{\fh}}
\nc{\fhb}{\hat{\fb}}
\nc{\hrho}{\hat{\rho}}

\nc{\hsl}{\hat{\fsl}}
\nc{\fpo}{\mathfrak{po}}
\nc{\dirlim}{\underset{\rightarrow}{\lim}\,}
\nc{\nen}{\newenvironment}
\nc{\ol}{\overline}
\nc{\ul}{\underline}
\nc{\ra}{\rightarrow}
\nc{\lra}{\longrightarrow}
\nc{\Lra}{\Longrightarrow}
\nc{\bo}{\bar{1}}
\nc{\Lla}{\Longleftarrow}

\nc{\Llra}{\Longleftrightarrow}

\nc{\thla}{\twoheadleftarrow}

\nc{\lang}{(}
\nc{\rang}{)}

\nc{\hra}{\hookrightarrow}

\nc{\iso}{\overset{\sim}{\lra}}

\nc{\ssubset}{\underset{\not=}{\subset}}

\nc{\vac}{|0\rangle}

\nc{\simka}{{\ \scriptscriptstyle _{{\sim}}^\text{\tiny{k}}\ }}

\nc{\Thm}[1]{Theorem~\ref{#1}}
\nc{\Prop}[1]{Proposition~\ref{#1}}
\nc{\Lem}[1]{Lemma~\ref{#1}}
\nc{\Cor}[1]{Corollary~\ref{#1}}
\nc{\Conj}[1]{Conjecture~\ref{#1}}
\nc{\Claim}[1]{Claim~\ref{#1}}
\nc{\Defn}[1]{Definition~\ref{#1}}
\nc{\Exa}[1]{Example~\ref{#1}}
\nc{\Rem}[1]{Remark~\ref{#1}}
\nc{\Note}[1]{Note~\ref{#1}}
\nc{\Quest}[1]{Question~\ref{#1}}
\nc{\Hyp}[1]{Hypoth\`ese~\ref{#1}}
\nen{thm}[1]{\label{#1}{\bf Theorem.\ } \em}{}
\nen{prop}[1]{\label{#1}{\bf Proposition.\ } \em}{}
\nen{lem}[1]{\label{#1}{\bf Lemma.\ } \em}{}
\nen{cor}[1]{\label{#1}{\bf Corollary.\ } \em}{}
\nen{conj}[1]{\label{#1}{\bf Conjecture.\ } \em}{}

\nen{claim}[1]{\label{#1}{\bf Claim.\ } \em}{}

\nen{defn}[1]{\label{#1}{\bf Definition.\ } }{}
\nen{exa}[1]{\label{#1}{\bf Example.\ } }{}
\nen{rem}[1]{\label{#1}{\em Remark.\ } }{}
\nen{note}[1]{\label{#1}{\em Note.\ } }{}
\nen{exer}[1]{\label{#1}{\em Exercise.\ } }{}
\nen{sket}[1]{\label{#1}{\em Sketch of proof.\ } }{}
\nen{quest}[1]{\label{#1}{\bf Question.\ } \em}{}

\nen{hyp}[1]{\label{#1}{\bf Hypoth\`ese.\ } \em}{}
\begin{document}
\setcounter{section}{0}
\title{On simplicity of universal minimal $W$-algebras}
\author{Maria Gorelik}
\address{M.G.: Department of Mathematics, Weizmann Institute of Science, Rehovot 761001, Israel; 
maria.gorelik@weizmann.ac.il}
\author{Victor G.~Kac}
\address{V. K.:  Department of Mathematics, MIT, 77 Mass. Ave, Cambridge, MA 02139;
kac@mit.edu}

\begin{abstract} 
	Simplicity of universal minimal quantum affine W-algebras is
studied. As an application, we find the values of the center, for which the
vacuum module over a
superconformal algebra is irreducible.
\end{abstract}


\medskip


%
%
%

%
%

\maketitle

\section{Introduction}
In the present paper we study the simplicity of universal minimal quantum affine $W$-algebras $W^k_{\min}(\dot{\fg})$ that we
began in~\cite{GKvac}. Here $\dot{\fg}$ is a finite-dimensional simple Lie superalgebra  over $\mathbb{C}$
with a reductive even part $\dot{\fg}_{\ol{0}}$, admitting a non-degenerate invariant bilinear form $(.\, ,.\, )$
whose restriction to $\dot{\fg}_{\ol{0}}$ is non-degenerate. 
Recall that $W^k_{\min}(\dot{\fg})$ is  a vertex algebra obtained by the quantum Hamiltonian reduction, associated to a minimal
$\fsl_2$-subalgebra of $\dot{\fg}_{\ol{0}}$, of the universal affine vertex algebra 
$V^k=V^k(\dot{\fg})$~\cite{KW}, where $V^k$ is the vacuum module of level $k\in\mathbb{C}$
over the affinization $\fg$  of  $\dot{\fg}$. Recall that a minimal
$\fsl_2$-subalgebra $\fs=span\{f,x,e\}$, where $[x,f]=-f$, $[x,e]=e$, $[e,f]=x$ of $\dot{\fg}$, is defined
by the property that the $\ad x$-eigenspaces decomposition of $\dot{\fg}$ has the form 
\begin{equation}\label{5.1}
\dot{\fg}=\dot{\fg}_{-1}\oplus \dot{\fg}_{-1/2}\oplus \dot{\fg} _0\oplus \dot{\fg}_{1/2}\oplus \dot{\fg}_{1},\ \ \text{ where }\ \
\dot{\fg}_{1}=\mathbb{C}e,\ \ \dot{\fg}_{-1}=\mathbb{C}f.\end{equation}
In this case the adjoint orbit of $e$ in $\dot{\fg}_{\ol{0}}$ has minimal non-zero dimension in its simple component,
hence the name ``minimal''.
As in~\cite{KW}, we normalize the bilinear form $(.\, ,.\, )$ by the condition $(\theta,\theta)=2$, 
where $\theta$ is  the root of $e$. Let $h^{\vee}$ be the dual Coxeter number of $\dot{\fg}$
(=$\frac{1}{2}$ eigenvalue of the Casimir element in the adjoint representation). Its values can be found in [KW2, Tables 1-3].

Let $\dot{\fh}$ be a Cartan subalgebra of the even part of $\dot{\fg}_0$ (=Cartan
subalgebra of $\dot{\fg}$), let $\dot{\Delta}\subset \dot{\fh}^*$ be the set of roots of $\dot{\fg}$, and
$\dot{\Delta}_{\ol{0}}$ be the subset of even roots. Set
$$\dot{\Delta}^{\#}:=\{\alpha\in \dot{\Delta}_{\ol{0}}|\ (\alpha,\alpha)\in\mathbb{R}_{>0}\}.$$
We call
the minimal $\fsl_2$-subalgebra $\fs$ of $\dot{\fg}$, the corresponding normalization of $(.\,,.\,)$, and the 
corresponding minimal $W$-algebra $W^k_{\min}(\dot{\fg})$ of  {\em unitary type} if
$\dot{\Delta}^{\#}=\{\pm\theta\}$, and of {\em  non-unitary type} otherwise.
The minimal $W$-algebra depends on the choice of $\fs$ (so, a more adequate notation for it would be
$W^k_{\min}(\dot{\fg},\fs)$), but the type determines it uniquely.

The gradings~(\ref{5.1}) for all minimal $\fs$ in $\dot{\fg}$ are listed in ~[KW2, Tables 1-3], and those of unitary type are contained in Table 2.
Recall that Table 2 is characterized by the property that the decompostion~(\ref{5.1}) 
is consistent with parity of $\dot{\fg}$, i.e. 
$p(\dot{\fg}_j)\equiv 2j\ \mod 2$. However there are two cases in Table 2, which are not of unitary type:
they are $\dot{\fg}=\fosp(4|m)$ with even $m>2$ and $D(2|1,a)$  with $a\not\in\mathbb{R}_{>0}$.
All $\dot{\fg}$ with minimal $\fs$ of unitary type are as follows ($m\in\mathbb{Z}_{\geq 0}$):
\begin{equation}\label{list1}
\fsl(2|m)\ \text{ with } m\not=2, \ \mathfrak{spo}(2|m),\ \mathfrak{psl}(2|2),   \ D(2|1,a)\ \text{ with } a\in\mathbb{R}_{>0}, \ G(3),\ F(4).
\end{equation}
According to~\cite{KMFP},  the vertex algebra $W^k_{\min}(\dot{\fg})$ for some $k\not=-h^{\vee}$ is unitary  if and only if
the minimal $\fs$ is of unitary type (except that $a\in\mathbb{Q}$). Throughout the paper, as in~\cite{KW},
we use notation $\mathfrak{spo}(m|n)$ (resp., $\fosp(m|n)$) if the bilinear form 
$( .\, , .\,)$ is normalized in such a way that the square length of even roots of $\mathfrak{sp}_m$
(resp., of $\mathfrak{so}_m$) are positive; for $\fsl(m|n)$ they are positive for $\fsl_m$.
For $\dot{\fg}=D(2|1,a)$, $a\not=0,-1$,  the bilinear form 
$( .\, , .\,)$ is normalized such that the
square lengths
of even roots are $2, -2/(1+a)$, and $-2a/(1+a)$.

By~\cite{GKvac}, Thm. 9.1.2, 
the minimal $W$-algebra $W^k_{\min}(\dot{\fg})$ with $k\not=-h^{\vee}$  is simple if and only if either $V^k$ 
is an irreducible $\fg$-module (equivalently, $V^k$ is a simple vertex algebra), or $k\in\mathbb{Z}_{\geq 0}$
and the $\fg$-module $V^k$ has length two (equivalently, $V^k$ has a unique non-zero proper submodule).

The criterion of irreducibility in terms of $k$ of the $\fg$-module $V^k$ was established in~\cite{GKvac}, \cite{HRvac}: 
\begin{itemize}
\item if  $\dot{\fg}$ is a Lie algebra, then
$V^k$ is irreducible if and only if 
$l(k+h^{\vee})\not\in \mathbb{Q}_{\geq 0}\setminus\{\frac{1}{m}\}_{m=1}^{\infty}$
where $l$ is the lacety of $\dot{\fg}$ ($l=3$ for $G_2$, $l=2$ for $B_n,C_n, F_4$, $l=1$ otherwise);
\item if $\dot{\fg}=\fosp(1|2n)$, then
$V^k$ is irreducible if and only if 
$2(k+h^{\vee})\not\in \mathbb{Q}_{\geq 0}\setminus\{\frac{1}{2m-1}\}_{m=1}^{\infty}$;
\item in all other cases (i.e., when $\dot{\fg}$ has non-zero defect), 
$V^k$ is irreducible if and only if $\frac{k+h^{\vee}}{(\alpha,\alpha)}\not\in\mathbb{Q}_{\geq 0}$  for all
roots $\alpha$ of $\dot{\fg}_{\ol{0}}$.
\end{itemize}

In the present paper we study the case $k\in\mathbb{Z}_{\geq 0}$ (then $V^k$ is not irreducible).
 In the first part (Section~\ref{smallones}) we study the most interesting from physics 
viewpoint unitary type cases.

We prove that for $\dot{\fg}$ as in~(\ref{list1}), with the normalization of $( .\, ,. \,)$ of unitary type, 
and $k\in\mathbb{Z}_{\geq 0}$, the $\fg$-module $V^k$ has length $2$, provided that the following inequality holds
\begin{equation}\label{3}
k\geq \frac{5}{2}-2h^{\vee}.\end{equation}
It follows that for $k\in\mathbb{Z}_{\geq 0}$ the vertex algebra  $W^k_{\min}(\dot{\fg})$ is simple  in the following cases
$$
\begin{array}{ |c||c|c| c|c|c|c|} 
 \hline
 \dot{\fg} &\fsl_2,\mathfrak{spo}(2|1)  & \fsl(2|m), m\not=2 &   \mathfrak{spo}(2|m), m\geq 3 & 
 \mathfrak{psl}(2|2), D(2|1;a) &  G(3) & F(4) \\
\hline
h^{\vee} & 2,\frac{3}{2} & 2-m & 2-\frac{m}{2} & 0, 0& -\frac{3}{2} & -2\\
\hline
 k  &  k\in\mathbb{Z}_{\geq 0} & k\in\mathbb{Z}_{\geq 2m-1}&  k\in\mathbb{Z}_{\geq m-1}& k\in\mathbb{Z}_{\geq 3}&  k\in\mathbb{Z}_{\geq 6}&  k\in\mathbb{Z}_{\geq 7}\\
 \hline
\end{array}
$$

%


Recall that the central charge $c$ of $W^k_{\min}(\dot{\fg})$ is given by the following formula~\cite{KW}
$$c=c(k)=\frac{k\sdim\dot{\fg}}{k+h^{\vee}}-6k+h^{\vee}-4.$$

As a corollary we obtain the following results on irreducibility of vacuum modules
over superconformal algebras with central charge $c$, which is an improvement and correction of~\cite{GKvac}, 
Corollary 9.1.5.

(Recall that a superconforml algebra $\fg$ is a central extension of a simple Lie superalgebra, generated
by a Virasoro subalgebra and coefficients of odd fields of conformal weight $\frac{1}{2}$, of 
 even fields of conformal weight $1$, and $N$ odd fields
of confomral weight $\frac{3}{2}$;  the vacuum module over $\fg$ is the module, induced from the one-dimensional module $\mathbb{C}_c$ of its subalegbra, generated by coefficients with non-negative indices and the center, which acts by $c$.)

\subsection{}
\begin{cor}{}
\begin{enumerate}
\item For $N=0$ and $1$ this module is irreducible if and only if $c$ is not the central charge  of minimal models
of the corresponding simple vertex algebra.
\item For $N=2$ this module is irreducible if and only if $c$ is not of the form $3-\frac{6p}{q}$, where
$p$ and $q$ are coprime positive inetegers and $q\geq 2$ (the subset with $p=1$ coincides with the well-known set of central charges of $N=2$ minimal models; the subset with $p\geq 1$ is the set of
``admissible'' central charges for which the modified characters and  supercharacters form a modular invariant family~\cite{KW3}), except possibly, if $c=-3$ or $3$.

\item For $N=3$ this module is irreducible if and only if $c$ is not a rational number, except, possibly , if
$c=-3$ or $-9$.
\item For $N=4$ this module is irreducible if and only if $c$ is not a rational number, except, possibly , if
$c=-6,-12$ or $-18$.
\item For big $N=4$ this module is irreducible if and only if $c$ does not lie in one of the following sets:
$\mathbb{Q}_{\geq 0}$, $a\mathbb{Q}_{>0}$, $ -(a+1)\mathbb{Q}_{>0}$, $\frac{1}{a}\mathbb{Q}_{>0}$,
 $-\frac{1}{a+1}\mathbb{Q}_{>0}$,  $-\frac{a+1}{a}\mathbb{Q}_{>0}$,  $-\frac{a}{a+1}\mathbb{Q}_{>0}$, except, possibly,
 if $c=-6$ or $-12$.
\end{enumerate}
\end{cor}
\begin{proof}
When $k\not=-h^{\vee}$, the irreducibility of the vacuum module over $N=0,1,2$ superconformal algebras with central charge $c=c(k)$ is equivalent to simplicity of the vertex algebra $W^k_{\min}(\mathfrak{spo}(2|N)$;
over $N=4$ superconformal algebra  to simplicity of  $W^k_{\min}(\mathfrak{psl}(2|2)$; over $N=3$ superconformal algebra  with central charge $c(k)+\frac{1}{2}$ to simplicity of  $W^k_{\min}(\mathfrak{spo}(2|3)$; and over big $N=4$ superconformal algebra with central charge $c(k)+3$   to simplicity of  $W^k_{\min}(D(2|1;a))$ (see~\cite{KW}, Section 8). 

For $N=0,1$, $c(k)$ with $k\not=-h^{\vee}$ takes all values in $\mathbb{C}$. For $N=3$ and big $N=4$ superconformal algebras, $c(-h^{\vee})=0$, but for $c=0$ the vacuum module is not irreducible. 
Now the Corollary follows from~(\ref{3}) and the above Table.
\end{proof}

Our main theorem in unitary type cases is the following.

\subsection{}
\begin{thm}{thm2}
If the normalization of the bilinear form $(.\, , .\, )$ is of unitary type, then the $\fg$-module  $V^k$ with $k\in\mathbb{Z}_{\geq 0}$
has length $2$ if~(\ref{3}) holds.
\end{thm}

The results, displayed in the above Table, follow from this Theorem. 

In the second part of this paper (Sections~\ref{sect4} and~\ref{jantzen}), 
we study the cases of normalization of the bilinear form $(.\, , .\, )$  of non-unitary type.
Using Janzen's filtration and character formulas, we find lower bounds on $k\in\mathbb{Z}_{\geq 0}$, for
which the $\fg$-module $V^k$ has length greater than $2$. Namely, we prove the following
theorem.

\subsection{}
\begin{thm}{thm3}
For $k\in\mathbb{Z}_{\geq 0}$ the vacuum 
$\fg$-module $V^k$ has length greater than $2$ if 
\begin{enumerate}
\item $\dot{\fg}$ has defect $0$ (i.e., $\dot{\fg}$ is a Lie algebra or $\dot{\fg}=\mathfrak{spo}(2n|1)$), unless
$\dot{\fg}=\fsl_2$ or  $\dot{\fg}=\mathfrak{spo}(2|1)$, in which case $V^k$ has length $2$ for all 
$k\in\mathbb{Z}_{\geq 0}$;
\item
$\dot{\fg}=\fsl(m|n)$,
$\fosp(2m+1|2n)$ with $m-2\geq n>0$, or
$\dot{\fg}=\fosp(2m|2n)$ with $m-2\geq n>0$, or $\dot{\fg}=\mathfrak{spo}(2n|2m)$, $\mathfrak{spo}(2n|2m)$
with $n>m>0$, or non-unitary type cases of
 $\dot{\fg}=G(3)$,   $F(4)$;
 \item  $\dot{\fg}=\fsl(n+1|n), \fosp(2n+3|2n)$, $n>0$,  if $k\geq 1$;
 \item  $\dot{\fg}=\mathfrak{psl}(n|n)$ for $n>2$ or 
$\dot{\fg}=\mathfrak{spo}(2n|2n)$, $\mathfrak{spo}(2n|2n+1)$, $\fosp(2n+2|2n)$ 
with $n\geq 2$  if  $k\geq 2n$  for
  $\dot{\fg}=\fosp(2n+2|2n)$  and $k\geq n$ for other cases;
  \item $\dot{\fg}=\mathfrak{spo}(2n|2n+2)$ and $k>2n$;
  \item  $\dot{\fg}=\fosp(4|n)$ for even $n\geq 4$ if 
$k\in\mathbb{Z}_{\geq \frac{n}{2}}$;
  \item  $\dot{\fg}=D(2|1,a)$ of non-unitary type with $a<-1$ if $k\in(\mathbb{Z}_{>0}\cap a^{-1}\mathbb{Z})$. 
\end{enumerate}
\end{thm}

\subsection{}
For the case when $\dot{\fg}$ is a Lie algebra (i) is proved in~\cite{GKvac}, Theorem 9.1.2 (ii);
 in Section~\ref{defectzero} of the present paper we give a shorter proof of (i).

In Section~\ref{critical} we recall
 a proof that $V^{-h^{\vee}}$ is of infinite length (hence the vertex algebra $W^{-h^{\vee}}_{\min}(\dot{\fg})$ is never simple).

\subsection{}Open cases of Theorems~\ref{thm2} an~\ref{thm3}  for  $k\in\mathbb{Z}_{\geq 0}$, $k\not=-h^{\vee}$,
 are the following:
\begin{enumerate}
\item normalization of type  1 for $k<\frac{5}{2}-2h^{\vee}$, 
\item normalization of non-unitary type for $h^{\vee}<0$: $\dot{\fg}=\fsl(m|n)$ with $2<m<n$, $\dot{\fg}=\mathfrak{spo}(n|m)$ with $m>n+2\geq 4$,
$\dot{\fg}=\fosp(m|n)$ with $5\leq m<n+2$ for all $k$;  $\dot{\fg}=\fosp(4|n)$ for even $n\geq 4$ if $k<\frac{n}{2}$;
\item normalization of non-unitary type for $h^{\vee}=0$: $\dot{\fg}=\mathfrak{spo}(2n|n+2)$ for $k\leq 2n$;
$\dot{\fg}=\mathfrak{psl}(n|n)$  for $n>2$ if $k<n$;
$\dot{\fg}=\fosp(2n+2|2n)$ if  
$k<2n$;  $D(2|1,a)$ with $a<-1$ if $ak\not\in\mathbb{Z}$.
\item  normalization of non-unitary type for $h^{\vee}>0$: $\dot{\fg}=\fsl(n+1|n), \fosp(2n+3|2n)$  if  $k=0$ and 
$\dot{\fg}=\mathfrak{spo}(2n|2n)$, $\mathfrak{spo}(2n|2n+1)$    if $k<n$.
\end{enumerate}

{\em Acknowledgment.}
We are grateful to  Dra\v zen Adamovi\'c for  important suggestions and observations. V.K. was partially supported by Simons collaboration grant and by ISF Grant 1957/21.
M. G. was supported by  ISF Grant 1957/21.

\section{Preliminaries}

We will use the standard notation of~\cite{KLie} for Lie superalgebras and their root systems.

\subsection{Notation}
Let $\fg$ be the (non-twisted) affinization of $\dot{\fg}$ and $\Delta$ be the   root system
of $\fg$. The root system 
$${\Delta}^{\#}:=\{\alpha\in {\Delta}_{\ol{0}}|\ (\alpha,\alpha)\in\mathbb{R}_{>0}\}$$
is the affinization of $\dot{\Delta}^{\#}$.

We fix the triangular decomposition of $\fg_{\ol{0}}$ with the Cartan subalgebra $\fh$, containing $\dot{\fh}$.
It extends to a triangular decompostion of $\fg$ and restricts to that of $\dot{\fg}+\fh$ and $\dot{\fg}_{\ol{0}}+\fh$.
We denote by $\Sigma(\fg_{\ol{0}})$, $\Sigma:=\Sigma(\fg)$, $\Sigma(\dot{\fg}_{\ol{0}})$ and $\dot{\Sigma}:=\Sigma(\dot{\fg})$ 
the corresponding sets of simple roots. For $\lambda\in\fh^*$ we denote 
  by  $M_{\fg_{\ol{0}}}(\lambda)$,  $M(\lambda)$, $M_{\dot{\fg}_{\ol{0}}}(\lambda)$,
  and $M_{\dot{\fg}}(\lambda)$,
the corresponding Verma modules, respectively, and by
$L_{\fg_{\ol{0}}}(\lambda)$, $L(\lambda)$,
$L_{\dot\fg_{\ol{0}}}(\lambda)$,  and $L_{\dot{\fg}}(\lambda)$,their irreducible quotients.

We denote by $\rho$ a Weyl vector of ${\fg}$ and consider the usual shifted action of the Weyl group $W$  
of ${\fg}$ given by
$$w.\lambda=w(\lambda+\rho)-\rho.$$

\subsubsection{}
We set $\alpha_0:=\delta-\theta$ and let $s_0\in W$
be the reflection with respect to $\alpha_0$.

\subsubsection{}
For the fixed triangular decomposition of $\Delta$ we
let $\dot{\Delta}^+$, $\Delta^+$ be  sets of positive roots
for $\dot{\Delta}$ and $\Delta$ respectively. 
We set
$Q:=\mathbb{Z}\Delta$ and $Q^+:=\mathbb{Z}_{\geq 0}\Delta^+$; we  define $\dot{Q},\dot{Q}^+$ 
in a similar way.
We will often decompose $\mu\in Q$ as   $\mu=j\delta+\dot{\mu}$ where
$j\in\mathbb{Z}$, $\dot{\mu}\in\dot{Q}$.

\section{Proof of~\Thm{thm2}}\label{smallones}

We list the main steps of the proof of ~\Thm{thm2} in Section~\ref{plan} below.
The details occupy the rest of this section.

\subsection{Notation}
In what follows  the triangular decompositions for subalgebras of $\fg_{\ol{0}}$ are always  induced by the  triangular decomposition of $\fg_{\ol{0}}$ and  we consider each subalgebra equipped by
the invariant bilinear form induced by the bilinear form $(.\, ,.\, )$
of unitary type on $\dot{\fg}$. For subalgebras $\fl$, containing $\fh$, we
 will use notation $\Delta(\fl)$, $\Sigma(\fl)$
for the set of roots  and  the set of simple roots respectively.

\subsubsection{}\label{fl}
Recall that $\fg_{\ol{0}}$ is the affinization of $\dot\fg_{\ol{0}}$ which corresponds to the bilinear form $(.\, ,.\, )$. 
 
 The Lie algebra $\dot\fg_{\ol{0}}$ is the direct product of several  subalgebras $\dot{\fl}^{(1)},\ldots,\dot{\fl}^{(s)}$ ($s$ is $1$, $2$ or $3$) where $\dot{\fl}^{(i)}$ is  a simple Lie algebra for $i>1$ and $\dot{\fl}^{(1)}=\dot{\fg}^{\#}$ is either simple or 
 a product of simple and $\mathbb{C}$. 
 Let 
$${\fl}^{(i)}=\bigl(\dot{\fl}^{(i)}\otimes \mathbb{C}[t^{\pm 1}]\bigr)\oplus \mathbb{C}K_i\oplus \mathbb{C}d_i$$
be the affinization of $\dot{\fl}^{(i)}$ which corresponds to the restriction of   $( .\,, .\,)$ to $\dot{\fl}^{(i)}$; we view $\fl^{(i)}$ as a subalgebra of $\fg_{\ol{0}}$ (identifying $K_i$ with $K$ and $d_i$ with $d$) and call such subalgebra an {\em affine component} of $\fg_{\ol{0}}$;  we denote $\fg^{\#}:=\fl^{(1)}$.
We let $\dot{\fh}^{(i)}:=\fh\cap \dot{\fl}^{(i)}$ be the corresponding Cartan subalgebra of $\dot{\fl}^{(i)}$; then $\dot{\fh}=\oplus_{i=1}^{s} \dot{\fh}^{(i)}$.
We identify the Cartan subalgebra of $\fl^{(i)}$ with
 $\dot{\fh}^{(i)}\oplus  \mathbb{C}K\oplus \mathbb{C}d$.

%
%
%

\subsubsection{}
Let $\fl$ be an affine component of $\fg_{\ol{0}}$ and $W(\fl)$ be its Weyl group. 
Recall that the Weyl group $W$ of $\fg_{\ol{0}}$ and $\fg$ is generated by the reflections $s_{\alpha}$ for 
$\alpha\in\Sigma(\fg_{\ol{0}})$ and $W=\prod\limits_{i=1}^{s} W(\fl^{(i)})$.   Let us check that the formula
$$w\circ\lambda=w(\lambda+\rho_{\fl^{(i)}})-\rho_{\fl^{(i)}}\ \text{  for }w\in W(\fl^{(i)}),\ \ i=1,\ldots,s,$$ 
defines an action of $W$ on $\fh^*$.  Note that
$$s_{\alpha}\circ \lambda=\lambda-\alpha-\frac{(\lambda,\alpha)}{(\alpha,\alpha)}\alpha\ \ \text{
for each }\alpha\in\Sigma(\fg_{\ol{0}}).$$
Since for $w\in W(\fl^{(i)})$ the action $\circ$ is the usual
shifted action of $W(\fl^{(i)})$ on $\fh^*$, it is enough  to verify that $s_{\alpha}\circ (s_{\beta}\circ \lambda)=s_{\beta}\circ 
(s_{\alpha} \circ \lambda)$
for $\alpha\in \Sigma(\fl^{(i)})$ and $\beta\in  W(\fl^{(j)})$ with $i\not=j$. In this case $(\alpha,\beta)=0$ and the above formula 
gives
$$s_{\beta}\circ (s_{\alpha}\circ \lambda)=s_{\beta}\circ (\lambda-\alpha-\frac{(\lambda,\alpha)}{(\alpha,\alpha)}\alpha)
=\lambda-\alpha-\frac{(\lambda,\alpha)}{(\alpha,\alpha)}\alpha-\beta-\frac{(\lambda,\beta)}{(\beta,\beta)}\beta=
s_{\alpha}\circ (s_{\beta}\circ \lambda)$$
as required. For $\alpha\in \Sigma\cap \Sigma(\fl^{(i)})$ we have
$$s_{\alpha}\circ\lambda=s_{\alpha}\lambda+s_{\alpha} \rho_{\fl^{(i)}}-\rho_{\fl^{(i)}}=s_{\alpha}\lambda-\alpha=
s_{\alpha}.\lambda.$$
Thus  $s_{\alpha}\circ\lambda=s_{\alpha}.\lambda$ if $\alpha\in \Sigma\cap \Sigma(\fg_{\ol{0}})$.

We will use the following simple formula verified in~\ref{circle}:
 \begin{equation}
 \label{wcirceta}
 ||\eta+\rho||^2-||w\circ\eta+\rho||^2=2(\rho-\rho_{\fl}, \eta-w\circ\eta) \ \text{ for } 
 w\in W(\fl)\ \text{ and }\eta\in\fh^*.
 \end{equation}

 \subsubsection{}
\begin{defn}{}
A weight $\lambda\in\fh^*$ is called {\em $\dot\fg_{\ol{0}}$-dominant} if  $L_{\dot\fg_{\ol{0}}}(\lambda)<\infty$.
\end{defn}

\subsubsection{}
Take $\lambda\in\fh^*$.
We denote by $N_{\fg_{\ol{0}}}(\lambda)$ the maximal quotient  of $M_{\fg_{\ol{0}}}(\lambda)$ which is
locally finite as $\dot\fg_{\ol{0}}$-module. Clearly, $N_{\fg_{\ol{0}}}(\lambda)\not=0$
if and only if $\lambda$ is $\dot\fg_{\ol{0}}$-dominant.

Let $\fl$ be an affine component of $\fg_{\ol{0}}$. For $\lambda'\in (\fh\cap\fl)^*$ 
we denote by $N_{\fl}(\lambda')$ the maximal quotient  of 
$M_{\fl}(\lambda')$ which is
locally finite as $\dot{\fl}$-module. 

Assume that $(\lambda+\rho_{\fl},\delta)\not=0$ for
all affine components $\fl$ of $\fg_{\ol{0}}$. Note that $k\Lambda_0$ satisfies this assumption if 
$k\geq \frac{5}{2}-2h^{\vee}$ if $\dot{\fg}$ is from the list~(\ref{list1}).

It is well known that each subquotient of $M_{\fg_{\ol{0}}}(\lambda)$ (resp., 
$M_{\fl}(\lambda')$) is isomorphic to $L_{\fg_{\ol{0}}}(w\circ \lambda)$ 
(resp., $L_{\fl}(w\circ \lambda')$) for some $w\in W$
(resp., $w\in W(\fl)$). For  $\lambda\in \fh^*$ and $\lambda'\in (\fh\cap\fl)^*$  we set
$$\begin{array}{ll}
\Howl_{\fg_{\ol{0}}}(\lambda):=\{w\in W|\ [N_{\fg_{\ol{0}}}(\lambda):L_{\fg_{\ol{0}}}(w\circ \lambda)]\not=0\},\\
\Howl_{\fl}(\lambda'):=\{w\in W(\fl)|\ [N_{\fl}(\lambda'):L_{\fl}(w\circ \lambda')]\not=0\},\\
\Howl_{\fl}(\lambda):=\Howl_{\fl}(\lambda|_{\fh\cap \fl}).
\end{array}$$

We will use the following formula which will be checked in~\Lem{lemMlLl}:
\begin{equation}\label{howl}
\Howl_{\fg_{\ol{0}}}(\lambda)=\prod\limits_{i=1}^{s}  \Howl_{\fl^{i)}}(\lambda),\ \ \lambda\in\fh^*.\end{equation}

\subsection{Proof of~\Thm{thm2}}\label{plan}
One has, using the Casimir operator,
$$[V^k:L(\lambda)]\not=0\ \ \Longrightarrow\ \ [V^k:L_{\fg_{\ol{0}}}(\lambda)]\not=0\ \text{ and }
||\lambda+\rho||^2=||k\Lambda_0+\rho||^2.$$

We will check that 
\begin{equation}\label{need}
[V^k:L_{\fg_{\ol{0}}}(\lambda)]\not=0\ \text{ and }
||\lambda+\rho||^2=||k\Lambda_0+\rho||^2\ \ \Longrightarrow\ \ \lambda\in\{k.\Lambda_0,
s_0.k\Lambda_0\}.\end{equation}

Assume that  $\alpha_0=\delta-\theta$ lies in  $\Sigma$. 
Then we have $[V^k:L(s_0.k\Lambda_0]=1$, so~(\ref{need}) implies that
$V^k$ has length two.

 Let $\cA$ be the set of finite subsets of $\Delta_{\ol{1}}^+\setminus\dot{\Delta}$. 
 We set 
$$\mu_A:=\sum\limits_{\scriptstyle{\alpha\in A}}\alpha\ \text{ 
for } A\in\cA,$$
and introduce the following multiset 
$$\cA_+:=\{A\in\cA\ |\  \mu_A\not=0,\ \ -\mu_A\ \text{ is $\dot{\fg}_{\ol{0}}$-dominant}\}.$$

Note that the sets  $\cA$ and $\cA_+$  are countable  and they
do not depend on the choice of
$\dot{\Sigma}$: even though the set $\Delta_{\ol{1}}^+$ depends on the choice of 
$\dot{\Sigma}$,
 the set $\Delta_{\ol{1}}^+\setminus\dot{\Delta}=\dot{\Delta}_{\ol{1}}+\mathbb{Z}_{>0}\delta$ does not depend on this choice.

In~\Cor{corfiltr}, we show that 
each $\fg$-module $V^k$ admits an increasing  filtration by $\fg_{\ol{0}}$-submodules
$0=N_0\subset N_1\subset N_2\subset \ldots$
such that $N_{j+1}/N_{j}$ is a 
quotient of $N_{\fg_{\ol{0}}}(k\Lambda_0-\mu_{A_j})$ and $\{A_j\}_{j=1}^{\infty}=\cA_+$.
This gives
$$[V^k:L_{\fg_{\ol{0}}}(\lambda)]\not=0\ \ \Longrightarrow\ \  \lambda=w\circ(k\Lambda_0-\mu_A)\ 
\ \text{ for } \ \ A\in\cA_+,\  w\in \Howl_{\fg_{\ol{0}}}(k\Lambda_0-\mu_A).$$

Therefore for the implication~(\ref{need}) it is enough to verify that
for all $A\in\cA_+$ one has
\begin{equation}\label{tempo}
||w\circ (k\Lambda_0-\mu_A)+\rho||^2\leq ||k\Lambda_0+\rho||^2 \text{ for }w\in \Howl_{\fg_{\ol{0}}}(k\Lambda_0-\mu_A)
\end{equation}
and the above inequality is strict except when $\mu_A=0$ and $w\circ(k\Lambda_0-\mu_A)\in\{k\Lambda_0, s_0.k\Lambda_0\}$.

We have $W=W^{\#}\times W_-$
where $W^{\#}:=W(\fg^{\#})$ is the subgroup generated by $s_{\alpha}$ with
$(\alpha,\alpha)>0$ and $W_-$ is the subgroup generated by $s_{\alpha}$ with
$(\alpha,\alpha)<0$. 
Take any $\eta\in\fh^*$ and $w\in \Howl_{\fg_{\ol{0}}}(\eta)$. Write $w=w_+w_-$ where $w_+\in W(\fg^{\#})$
and $w\in W_-$. From now on we assume that 
$\dot{\Sigma}$ contains all negative square length simple roots for $\dot\fg_{\ol{0}}$
(see Section~\ref{Step4} for the list of $\dot{\Sigma}$).
In Lemmas~\ref{lemcirc}, \ref{lemassmstar} 
we show that
\begin{equation}\label{w+<w_-}
||w\circ \eta+\rho||^2<||w_+\circ \eta+\rho||^2\ \text{ if } w_-\circ \eta\not=\eta.
\end{equation}
Hence it is enough to check~(\ref{tempo}) for 
$w\in \Howl_{\fg^{\#}}(k\Lambda_0-\mu_A)$.

Fix $k\in\mathbb{Z}_+$ with $k+h^{\vee}\not=0$.
Since $k\Lambda_0$ is $\fg^{\#}$-dominant we have
$$\Howl_{\fg^{\#}}(k\Lambda_0-j\delta)=\Howl_{\fg^{\#}}(k\Lambda_0)=\{\Id, s_0\}.$$
Since $\alpha_0\in\Sigma$, we have $w\circ\eta=w.\eta$ for any $\eta\in\fh^*$
and $w\in \Howl_{\fg^{\#}}(k\Lambda_0-j\delta)$.
Therefore
$$||w\circ (k\Lambda_0-j\delta)+\rho||^2=||w.(k\Lambda_0-j\delta)+\rho||^2
=||k\Lambda_0-j\delta+\rho||^2\leq ||k\Lambda_0+\rho||^2$$
and  the above inequality is strict  for $j>0$ (since $k+h^{\vee}>0$).
Hence~(\ref{tempo}) holds for $\mu_A\in\mathbb{Z}_{\geq 0}\delta$.

The above argument reduces (\ref{tempo})  to the following assertion:
for all  $A\in\cA_+$ such that $\mu_A\not\in \mathbb{Z}_{\geq 0}\delta$
 we have
\begin{equation}\label{eqstar}
0<||k\Lambda_0+\rho||^2-||w\circ (k\Lambda_0-\mu_A)+\rho||^2 \ \text{ for all }\ \ 
w\in \Howl_{\fg^{\#}}(k\Lambda_0-\mu_A).
\end{equation}
%

\subsubsection{}
\begin{lem}{lemcheck!}
Set
$\ u_A:=2(k+h^{\vee})(\Lambda_0,\mu_A)-(1-h^{\vee})(\mu_A,\theta)
-||\mu_A||^2$. Then for each $A\in\cA_+$ the inequalities 
\begin{equation}\label{check!}
 \begin{array}{lcl}
   0<u_A, & & 
 2(2-h^{\vee}) (-\mu_A,\theta) \leq u_A\end{array}\end{equation} 
imply~(\ref{eqstar}).
\end{lem}
\begin{proof}
Fix $A\in\cA$ and set 
$$\lambda_A:=k\Lambda_0-\mu_A,\ \  \ \ 
u:=||k\Lambda_0+\rho||^2-||\lambda_A+\rho||^2.$$

By~(\ref{wcirceta}) we have
$$
||k\Lambda_0+\rho||^2-||w\circ \lambda_A+\rho||^2=u+2(\rho-\rho^{\#}, \lambda_A-w\circ\lambda_A)$$
so~(\ref{eqstar}) can be rewritten as
$$2(\rho^{\#}-\rho, \lambda_A-w\circ\lambda_A)<u.$$

One has $u=2(k\Lambda_0+\rho,\mu_A)-||\mu_A||^2$.
Recall that $\rho=h^{\vee}\Lambda_0+\dot{\rho}$. 
 By~(\ref{lemeta}) for our choice of $\dot{\Sigma}$ 
 we have $(2\dot{\rho},-\mu_A)\geq (1-h^{\vee})(-\mu_A,\theta)$ 
 (since $-\mu_A$ is $\dot\fg_{\ol{0}}$-dominant for $A\in\cA_+$). Thus
$$2(k\Lambda_0+\rho,\mu_A)-||\mu_A||^2\geq 
2(k+h^{\vee})(\Lambda_0,\mu_A)-(1-h^{\vee})(\mu_A,\theta)
-||\mu_A||^2=u_A.$$
Hence $u\geq u_A$, so  it is enough to check that
\begin{equation}\label{rhoru}
2(\rho^{\#}-\rho, \lambda_A-w\circ\lambda_A)<u_A.
\end{equation}

We have $(\rho^{\#}-\rho,\alpha_0)=0$ and $(\rho^{\#}-\rho,\delta)=
2-h^{\vee}$. Since $(\Lambda_0+\frac{\theta}{2},\alpha_0)=0$,  
the vector $(\rho^{\#}-\rho)-
(2-h^{\vee})(\Lambda_0+\frac{\theta}{2})$ is orthogonal to $\alpha_0$ and to $\delta$.
Since  $w\in W(\fg^{\#})=W(\hat{\fsl}_2)$,  the element
 $\lambda_A-w\circ\lambda_A$ is a linear combination of $\delta$ and 
$\alpha_0$, so 
$$(\rho^{\#}-\rho, \lambda_A-w\circ\lambda_A)=(2-h^{\vee})(\Lambda_0+\frac{\theta}{2},  \lambda_A-w\circ\lambda_A).$$

Now assume that~(\ref{check!}) holds. The inequality $u_A>0$ gives~(\ref{rhoru})  for 
 $\lambda_A=w\circ\lambda_A$. 
 Consider the case when  $\lambda_A\not=w\circ\lambda_A$.
Then~\Cor{corsl2need}  from the Appendix gives
$$(\Lambda_0+\frac{\theta}{2}, \lambda_A-w\circ\lambda_A)<(\lambda_A,\theta).$$
Since  $h^{\vee}<2$ we obtain
$$(\rho^{\#}-\rho, \lambda_A-w\circ\lambda_A)<(2-h^{\vee})(\lambda_A,\theta)=
(2-h^{\vee})(-\mu_A,\theta).$$
Using the second inequality of~(\ref{check!}) we get
$$2(\rho^{\#}-\rho, \lambda_A-w\circ\lambda_A)<
2(2-h^{\vee})(-\mu_A,\theta)\leq u_A.$$
Thus~(\ref{check!}) forces~(\ref{rhoru}).
\end{proof}

In~\Cor{cor7}  we will show that~(\ref{check!}) holds if  $k+2h^{\vee}\geq \frac{5}{2}$.
This will complete the proof  of~\Thm{thm2}.

 \subsection{Remark}

The proof does not work for  $k+2h^{\vee}=2$
and we do not know whether $V^{2-2h^{\vee}}$ is of length two.
Indeed, take
$$A:=\{\delta+\beta|\ \beta\in \dot{\Delta}_1\ \text{ s.t. }(\beta,\theta)<0\}.$$

Since $(\beta,\theta)=\pm 1$ for all
 $\beta\in \dot{\Delta}_1$ we have
 $$\mu_A=l\delta-\frac{l}{2}\theta \ \ \text{ where }\ \ \ 
l:=\frac{\dim{\dot{\fg}}_1}{2}=4-2h^{\vee}$$
so $k+2=l=(-\mu_A,\theta)$. 
By~\Lem{lemsl2} from the Appendix, $\Howl_{\fg^{\#}}(k\Lambda_0-\mu_A)$ contains
$s_{2\delta-\theta}$. One has
$s_{2\delta-\theta}\circ (k\Lambda_0-\mu_A)=s_{2\delta-\theta}.k\Lambda_0$, 
in particular,
 $\ ||s_{2\delta-\theta}\circ (k\Lambda_0-\mu_A)+\rho||^2=||k\Lambda_0+\rho||^2$, so~(\ref{tempo})
 is not strict though $\mu_A\not=0$ and $w\circ (k\Lambda_0-\mu_A)\not\in\{k\Lambda_0, s_0.k\Lambda_0\}$.

\subsection{Filtration on $V^k$}\label{filter}
%
%
Set
$$\htt(\sum\limits_{\alpha\in\Sigma} n_{\alpha}\alpha):=\sum\limits_{\alpha\in\Sigma} n_{\alpha},
\ \ \cA^{(n)}:=\{A\in \cA|\ \htt \mu_A=n\}.$$ 

Note that 
$\mu_A\in\mathbb{Z}_{\geq 0}\Sigma$ for each $A\in\cA$. In particular, $\cA^{(n)}$ is empty
for $n\not\in\mathbb{Z}_{\geq 0}$ and $\cA^{(0)}=\{\emptyset\}$ and each  $\cA^{(n)}$ is finite. 
This allows to 
enumerate  $\cA$ in the following  way:
 we set $A_1:=\emptyset$, then 
we enumerate arbitrarily all elements in $\cA^{(1)}$, then all elements in $\cA^{(2)}$
and so on. For this enumeration
$\htt\mu_{A_j}>\htt\mu_{A_i}$ forces $ j>i$.
 Since $\nu-\nu'  \in \mathbb{Z}_{\geq 0}\Delta_{\ol{0}}^+$ implies
$\htt\nu>\htt\nu'$ or $\nu=\nu'$ we have 
\begin{equation}\label{enu}
\mu_{A_j}-\mu_{A_i}  \in \mathbb{Z}_{\geq 0}\Delta_{\ol{0}}^+\  \ \Longrightarrow\ \ j\geq i.\end{equation}

\subsubsection{}
For each $\beta\in\Delta_1$ we fix a non-zero element $f_{-\beta}\in\fg_{-\beta}$
and for each $A\in\cA$ we fix $f_A:=\prod\limits_{\beta\in A} f_{-\beta}$ where the product is taken
in any order. By~(\ref{enu})  for $e\in\fn^+_0$ we have
$$[e,f_{A_j}]\in \sum\limits_{i<j}  \cU(\fn^-_0) f_{A_i}+\cU(\fh+\fn^+).$$

Let $v_0$ be the highest weight vector of $V^k$. For $j=1,2\ldots$ let
$N_j$ be the $\fg_{\ol{0}}$-submodule generated by the vectors $f_{A_1}v_0,\ldots,f_{A_j}v_0$.
By above,
$$[e,f_{A_j}]v_0\in \sum\limits_{i<j}  \cU(\fn^-_0) f_{A_i}v_0$$
so the image of $f_{A_j}v_0$ is the primitive vector in 
$N_{j+1}/N_j$.
Thus $N_{j+1}/N_j$ is a quotient of 
$M_{\fg_{\ol{0}}}(k\Lambda_0-\mu_{A_j})$.
Since $V^k$ is $\dot\fg_{\ol{0}}$-integrable,  this quotient is
$\dot\fg_{\ol{0}}$-integrable, so $N_{j+1}/N_j$ is a quotient of $N_{\fg_{\ol{0}}}(k\Lambda_0-\mu_{A_j})$.

\subsubsection{}
\begin{cor}{corfiltr}
The module $V^k$ admits an increasing  filtration by $\fg_{\ol{0}}$-submodules
$$0=N_0\subset N_1\subset N_2\subset \ldots$$
where $N_{j+1}/N_j$ is a 
quotient of $N_{\fg_{\ol{0}}}(k\Lambda_0-\mu_{A_j})$. In particular, $N_{j+1}=N_j$ for $A_j\not\in\cA_+$.
\end{cor}
%
%
%
%
%

\subsection{Structure of $N_{\fg_{\ol{0}}}(\lambda)$}
One has  $N_{\fg_{\ol{0}}}(\lambda)=0$ if $\lambda$ is not $\dot\fg_{\ol{0}}$-dominant; if $\lambda$ is  
$\dot\fg_{\ol{0}}$-dominant, then
$$N_{\fg_{\ol{0}}}(\lambda)=M_{\fg_{\ol{0}}}(\lambda)/\sum\limits_{\alpha\in \Sigma(\dot\fg_{\ol{0}})} 
M(s_{\alpha}\circ \lambda).$$

\subsubsection{}
\begin{lem}{lemnu=0}
If $[N_{\fg_{\ol{0}}}(\lambda):L_{\fg_{\ol{0}}}(\lambda-\nu)]\not=0$ and $\nu\in \mathbb{Z}\Delta(\dot\fg_{\ol{0}})$,
then  $\nu=0$.
\end{lem}
\begin{proof}
If $\lambda$ is not $\dot\fg_{\ol{0}}$-dominant, then $N_{\fg_{\ol{0}}}(\lambda)=0$.
Assume that $\lambda$ is $\dot\fg_{\ol{0}}$-dominant.
Define a $\mathbb{Z}_{\geq 0}$-grading on $N_{\fg_{\ol{0}}}(\lambda)$ by letting the degree of
the vector of weight $\lambda-\mu$ be equal to $(\mu,\Lambda_0)$.
The homogeneous component of degree zero is 
$$M_{\dot\fg_{\ol{0}}}(\lambda)/\sum\limits_{\alpha\in \Sigma(\dot\fg_{\ol{0}})} 
M_{\dot\fg_{\ol{0}}}(s_{\alpha}\circ \lambda)\cong L_{\dot\fg_{\ol{0}}}(\lambda).$$
This implies the claim.
\end{proof}

\subsubsection{}
\begin{lem}{lemMlLl}
One has: $\Howl_{\fg_{\ol{0}}}(\lambda)=\prod\limits_{i=1}^{s}  \Howl_{\fl^{i)}}(\lambda)$ for all $\lambda\in\fh^*$.
\end{lem}
\begin{proof}
We retain notation of~\ref{fl} and set $\tilde{\fl}:=\prod\limits_{i=1}^{s}  \fl^{(i)}$.

Denote by $\fz$  the central ideal spanned by
$K_i-K_j$ for $i,j=1,\ldots,s$.
The  algebra $\fg_{\ol{0}}$ can be identified with the 
subalgebra of  $\tilde{\fl}/\fz$ which is spanned by the images of 
 $\dot{\fl}^{(i)}\otimes \mathbb{C}[t^{\pm 1}]\oplus \mathbb{C}K_i$
 for $i=1,\ldots,s$ and by $\sum\limits_{i=1}^{s} d_i$ which identifies
with $d$.

Note that $\tilde{\fh}:=\oplus_{i=1}^{s}  \fh^{(i)}$ is  the Cartan subalgebra of $\tilde{\fl}$.
Fix $\tilde{\lambda}\in \tilde{\fh}^*$ by setting 
$$\langle \tilde{\lambda},K_i\rangle:=\langle \lambda,K\rangle,\ \ \ 
\langle \tilde{\lambda},d_i\rangle:=\frac{\langle \lambda,d\rangle}{s},\ 
\ \ \langle \tilde{\lambda},h\rangle =\langle \lambda,h\rangle\ \ \text{
for }h\in \dot{\fh}.$$  The Verma module $M_{\tilde{\fl}}(\tilde{\lambda})$ 
is a module over $\tilde{\fl}/\fz$ and 
$M_{\fg_{\ol{0}}}(\lambda)=\Res^{\tilde{\fl}/\fz}_{\fg_{\ol{0}}}  M_{\tilde{\fl}/\fz}(\tilde{\lambda})$
so we can identify $M_{\fg_{\ol{0}}}(\lambda)$ with $M_{\tilde{\fl}}(\tilde{\lambda})$. 
In this manner $N_{\fg_{\ol{0}}}(\lambda)$ identifies with
 the maximal $\dot\fg_{\ol{0}}$-integrable quotient of $M_{\tilde{\fl}}(\tilde{\lambda})$, i.e. 
$$N_{\fg_{\ol{0}}}(\lambda)=\Res^{\tilde{\fl}/\fz}_{\fg_{\ol{0}}}M_{\tilde{\fl}}(\tilde{\lambda})/\sum\limits_{\alpha\in {\Sigma}(\dot{\fg_{\ol{0}}})} M_{\tilde{\fl}}(s_{\alpha}\circ \tilde{\lambda}).$$

Since $\tilde{\fl}=\prod\limits_{i=1}^{s }\fl^{(i)}$,  the Verma module
$M_{\tilde{\fl}}(\tilde{\lambda})$ can be identified with  the external tensor product
 $\boxtimes_{i=1}^s M_{\fl^{(i)}}(\tilde{\lambda}_i)$
where $\tilde{\lambda}_i$ is the restriction of $\tilde{\lambda}$ to $\fh^{(i)}$.
Thus $N_{\fg_{\ol{0}}}(\lambda)$ identifies with 
$$M_{\tilde{\fl}}(\tilde{\lambda})/\sum\limits_{\alpha\in \dot{\Sigma}_0} M_{\tilde{\fl}}(s_{\alpha}\circ \tilde{\lambda})\cong
\boxtimes_{i=1}^s\bigl( M_{\fl^{(i)}}(\tilde{\lambda}_i)/\sum\limits_{\alpha\in \Sigma(\dot{\fl}^{(i)})} M_{\tilde{\fl}}(s_{\alpha}\circ \tilde{\lambda}_i)\bigr)=\boxtimes_{i=1}^s N_{\fl^{(i)}}(\tilde{\lambda}_i) .$$
Therefore  $\Howl_{\fg_{\ol{0}}}(\lambda)=\prod\limits _{i=1}^s \Howl_{\fl^{(i)}}(\tilde{\lambda}_i)$.

Recall that we view
 $\fl^{(i)}$ as a subalgebra of $\fg_{\ol{0}}$ (by mapping $K_i$ to $K$ and $d_i$ to $d$).
One has 
 $\langle \tilde{\lambda}_i-\lambda_i,h\rangle=0$ for $h\in \dot{\fh}^{(i)}$ and
 $\langle \tilde{\lambda}_i-\lambda_i,K_i\rangle=0$. Thus
 $\tilde{\lambda}_i-\lambda_i$ is proportional to 
 the minimal  imaginary root in $\Delta(\fl^{(i)})$.
 Then $N_{\fl^{(i)}}$ is the tensor product
 of $N_{\fl^{(i)}}({\lambda}_i)$ and a one-dimensional module $\fl^{(i)}$-module so
 $$\Howl_{\fl^{(i)}}(\tilde{\lambda}_i)=\Howl_{\fl^{(i)}}(\lambda_i).$$
 By definition, $\Howl_{\fl^{(i)}}(\lambda)=\Howl_{\fl^{(i)}}(\lambda_i)$.
 This completes the proof.
\end{proof}

\subsection{Proof of~(\ref{wcirceta})}\label{circle}
Let $\fl$ be one of the affine components of $\fg_{\ol{0}}$. 

 For any $\eta\in\fh^*$ and $w\in W(\fl)$ we have
 $$\begin{array}{l}
 ||\eta+\rho||^2-||w\circ\eta+\rho||^2=\\
 (||\eta+\rho||^2-||\eta+\rho_{\fl}||^2)+(||\eta+\rho_{\fl}||^2-
 ||w\circ\eta+\rho_{\fl}||^2)+( ||w\circ\eta+\rho_{\fl}||^2- ||w\circ\eta+\rho||^2)\\
 =  (||\eta+\rho||^2-||\eta+\rho_{\fl}||^2)-( ||w\circ\eta+\rho||^2- ||w\circ\eta+\rho_{\fl}||^2)\\
 =(2(\rho-\rho_{\fl}, \eta)+||\rho||^2-||\rho_{\fl}||^2)-
 (2(\rho-\rho_{\fl}, w\circ\eta)+||\rho||^2-||\rho_{\fl}||^2)\\
 =2(\rho-\rho_{\fl}, \eta-w\circ\eta).
 \end{array}$$
 Hence 
 $||\eta+\rho||^2-||w\circ\eta+\rho||^2=2(\rho-\rho_{\fl}, \eta-w\circ\eta)$;
 this proves~(\ref{wcirceta}).

\subsection{Proof of~(\ref{w+<w_-})}
Recall that  $W=W^{\#}\times W_-$
where $W^{\#}:=W(\fg^{\#})$ is the subgroup generated by $s_{\alpha}$ with
$(\alpha,\alpha)>0$ and $W_-$ is the subgroup generated by $s_{\alpha}$ with
$(\alpha,\alpha)<0$. 
\subsubsection{}
\begin{lem}{lemcirc}
Assume that for each affine component $\fl\not=\fg^{\#}$  we have 
\begin{equation}\label{assmstar}
 (\rho-\rho_{\fl}, \alpha)\geq 0\ \ 
\text{ if }\alpha\in \Sigma(\dot{\fl})\ \text{ and } \ \ (\rho-\rho_{\fl}, \alpha)>0
\text{ if }\alpha\in \Sigma(\fl\setminus \Sigma(\dot{\fl})).
\end{equation}
Take any $\lambda\in\fh^*$ and $w\in \Howl_{\fg_{\ol{0}}}(\lambda)$. Write $w=w_+w_-$ where 
$w_+\in W(\fg^{\#})$
and $w\in W_-$.  Then
$$
||w\circ \lambda+\rho||^2<||w_+\circ \lambda+\rho||^2\ \text{ for } w_-\circ \lambda\not=\lambda .$$
\end{lem}
\begin{proof}
If $s=1$, then $W(\fg^{\#})=W$, so $w_-=\Id$ and $w_-\circ \lambda=\lambda$.
Consider the case $s>1$, that is $s=2,3$.
Recall that the affine components of $\fg_{\ol{0}}$ are denoted by $\fl^{(1)},\ldots,\fl^{(s)}$
where $\fl^{(1)}=\fg^{\#}$ and $s\leq 3$. 
Write  $w=w_1\ldots w_{s}$ where $w_i\in W(\fl^{(i)})$.   Since $w\in \Howl_{\fg_{\ol{0}}}(\lambda)$, the formula (\ref{howl}) gives
$w_i\in \Howl_{\fl^{(i)}}(\lambda)$. One has
$w_1=w_+$ and $w_-=w_2$ if $s=2$, $w_-=w_2w_3$ if $s=3$.

Set $\rho^{(i)}:=\rho_{\fl^{(i)}}$. By~\ref{circle} we have
\begin{equation}\label{nuka}
||w_1\circ\lambda+\rho||^2-||w_2w_1\circ\lambda+\rho||^2= 2(\rho-\rho^{(2)}, \nu)\end{equation}
for
$\nu:=w_1\circ\lambda-w_2 w_1\circ\lambda$. Since $w_1\in W(\fl^{(1)})$ and $w_2\in W(\fl^{(2)})$ we have
$$\nu=w_1\circ\lambda- w_1w_2 \circ\lambda=w_1(\lambda+\rho^{(1)})-
w_1(w_2\circ\lambda+\rho^{(1)})=w_1(\lambda-w_2\circ\lambda)=\lambda-w_2\circ\lambda.$$

Therefore $\nu=\lambda_2-w_2\circ\lambda_2$
where $\lambda_2:=\lambda|_{\fh\cap \fl^{(2)}}$.
 Since $w_2\in \Howl_{\fl^{(2)}}(\lambda)$, 
one has 
$$[N_{\fl^{(i)} }(\lambda_2):L_{\fl^{(2)}}( \lambda_2-\nu)]\not=0.$$

Assume that $w_2\circ\lambda\not=\lambda$, i.e. $\nu\not=0$. Using~\Lem{lemnu=0} we obtain
$\nu\in \mathbb{Z}_{\geq 0} \Sigma(\fl^{(2)})$ and $\nu\not\in \mathbb{Z}_{\geq 0} \Sigma(\dot{\fl}^{(2)})$.
The assumption~(\ref{assmstar}) forces
$(\rho-\rho^{(2)}, \nu)>0$. Using~(\ref{nuka}) we obtain 
$||w_1\circ\lambda+\rho||^2> ||w_2w_1\circ\lambda +\rho||^2$.
 This completes the proof  for the case $s=2$.
For $s=3$ the same argument gives
$||w_2w_1\circ\lambda +\rho||^2> ||w_3w_2w_1\circ\lambda+\rho||^2$
if  $w_3\circ \lambda\not=\lambda$. This completes the proof.
\end{proof}
%
%

%
%
%
%
%

\subsection{Properties of $\dot{\Sigma}$}\label{Step4}
We write $\Sigma(\dot\fg_{\ol{0}})=\Sigma(\dot\fg_{\ol{0}})_+\coprod \Sigma(\dot\fg_{\ol{0}})_-$, where
$$\Sigma(\dot\fg_{\ol{0}})_{\pm}:=\{\alpha\in\Sigma(\dot\fg_{\ol{0}})| \ \pm (\alpha,\alpha)>0\}.$$
Recall that, since our normalization is of  unitary type ,
\begin{equation}\label{ika}
\Sigma(\dot\fg_{\ol{0}})_+=\{\theta\}. 
\end{equation}
We will use the following assumptions
\begin{enumerate}
\item $\alpha_0=\delta-\theta\in\Sigma$;
\item $\Sigma(\dot\fg_{\ol{0}})_-\subset \dot{\Sigma}$.
\end{enumerate}

Note that $\dot{\Sigma}$ satisfying (ii) exists for any $\dot{\fg}$ (but for
$\dot{\fg}:=\fosp(3|2)$ with the normalization
$||\vareps_1||^2=2$ there is no $\Sigma$ satisfying (i) and (ii)). For all
$\dot{\fg}$  from Table 2 in~\cite{KW} the affine superalgebra $\fg$ admits $\Sigma$ satisfying the above assumptions:
$$\begin{array}{lll}
\fsl(2|m), \mathfrak{psl}(2|2) & \dot{\Sigma}=\{\vareps_1-\delta_1,\delta_1-\delta_2,\ldots, \delta_{n-1}-\delta_n,\delta_n-\vareps_2\}\ & \theta=\vareps_1-\vareps_2\\
\mathfrak{spo}(2|2m+1) &   \dot{\Sigma}=\{\delta_1-\vareps_1,\vareps_1-\vareps_2,\ldots, \vareps_{m-1}-\vareps_m,\vareps_m\}\  &  \theta=2\delta_1\\
\mathfrak{spo}(2|2m) &   \dot{\Sigma}=\{\delta_1-\vareps_1,\vareps_1-\vareps_2,\ldots, \vareps_{m-1}-\vareps_m,\vareps_{m-1}+\vareps_m\} \ & \theta=2\delta_1  \\
D(2|1,a) & \dot{\Sigma}=\{\vareps_1-\vareps_2-\vareps_3,2\vareps_2,2\vareps_3\}\ & \theta=2\vareps_1\\
G(3) & \dot{\Sigma}=\{\delta_1+\vareps_3, \vareps_2-\vareps_1,\vareps_1\}\ & \theta=2\delta_1\\
F(4) &  \dot{\Sigma}=\{\frac{\delta_1-\vareps_1-\vareps_2-\vareps_3}{2},
\vareps_1-\vareps_2, \vareps_2-\vareps_3,\vareps_3\}\ & \theta=\delta_1.
\end{array}$$

We have the following formula for all $\dot{\fg}$ from Table 2 in ~\cite{KW}  (see~\cite{KW},  (5.6))
\begin{equation}\label{strange}
\dim\dot\fg_{\ol{1}}+4 h^{\vee}=8.\end{equation}
Since $\dim\dot\fg_{\ol{1}}\geq 4$, we have $h^{\vee}\leq 1$ for all $\dot{\fg}$ in the above list.

\subsubsection{}
\begin{lem}{lemassmstar}
If $\Sigma$ satisfies (ii) above,  
then~(\ref{assmstar}) holds.
\end{lem}
\begin{proof}
Take an affine component $\fl\not=\fg^{\#}$.
For  $\alpha\in \Sigma(\dot{\fl})$ we have  $(\alpha,\alpha)<0$, so
$\alpha\in \Sigma$  and thus  $(\rho-\rho_{\fl},\alpha)=0$.
Take $\alpha\in \Sigma({\fl})\setminus\Sigma(\dot{\fl})$.
Then $\alpha=\delta-\theta'$ where $\theta'\in \Delta(\dot{\fl})$.
By above, $(\rho-\rho_{\fl},\theta')=0$, so 
$(\rho-\rho_{\fl},\delta-\theta')=(\rho-\rho_{\fl},\delta)$.
One has $(\rho,\delta)\geq 0$ and $(\rho_{\fl},\delta)<0$ since $\delta$ is a positive linear combination of roots in  $\Sigma({\fl})$. Therefore $(\rho-\rho_{\fl},\delta-\theta')>0$.   Hence~(\ref{assmstar}) holds.
\end{proof}

\subsubsection{}
Note that $\{\alpha\in\dot{\Delta}_0: (\alpha,\alpha)<0\}$ is the root subsystem in $\dot{\Delta}_0$;
let  $\dot{\rho}_-$ be  the Weyl vector for this root subsystem. Recall that the choice of $\rho$ is not unique,
however the values $(\rho,\eta)$ and $(\dot{\rho},\eta)$ do not depend on this choice if $\eta\in\mathbb{Z}\Delta$.

Recall that the Lie superalgebra $\dot{\fg}$ is called of type I (resp.,  II) if
its even part $\dot{\fg}_{\ol{0}}$ is
not semisimple (resp., is semisimple); recall that for type I case the
subalgebra $\dot{\fg}_{\ol{0}}$ is  reductive with $1$-dimensional center.
\subsubsection{}
\begin{lem}{mutheta}
If $\dot{\fg}$ is of type II and $\dot{\fg}$ satisfies~(\ref{ika}), then  for any  $\mu\in\mathbb{Z}\Delta$
we have  $2 (\mu,\mu)\leq (\mu,\theta)^2$.
\end{lem}
\begin{proof}
Recall that  $\Sigma(\dot\fg_{\ol{0}})=\{\theta\}\coprod \Sigma(\dot\fg_{\ol{0}})_-$.
Since $\dot{\fg}$ is of type II, one has  $\mathbb{Z}\Delta=\mathbb{Z}\Sigma({\fg}_{\ol{0}})$, so
 $\mu=a\delta+b\theta+\mu_-$ for 
$a,b\in\mathbb{Q}$ and 
$\mu_-\in \mathbb{Q}  \Sigma(\dot\fg_{\ol{0}})_-$.  One has $||\mu_-||^2\leq 0$ and $(\mu_-,\theta)=0$. Therefore
$2||\mu_A||^2\leq  4b^2=(\mu_A,\theta)^2$.
\end{proof}

\subsubsection{}
\begin{lem}{lemrho-}
If $\Sigma$ satisfies~(\ref{ika}) and  (i), (ii) in Section~\ref{Step4},  
then we can choose the Weyl vector $\rho:=h^{\vee}\Lambda_0+\dot{\rho}_-+\frac{h^{\vee}-1}{2}\theta$.
\end{lem}
\begin{proof}
Set 
$$\zeta:=h^{\vee}\Lambda_0+\dot{\rho}_-+\frac{h^{\vee}-1}{2}\theta-\rho.$$
It is enough to show that $(\zeta,\beta)=0$ for all $\beta\in\Delta$.
Using the assumption  (ii) we obtain 
 $(\zeta,\alpha)=0$ for   $\alpha\in\Sigma(\dot\fg_{\ol{0}})_-$ and 
 and $(\zeta,\delta)=0= (\zeta,\delta-\theta)=0$. 
 The assumption (i) gives
  \begin{equation}\label{rho'rho}
  (\zeta,\alpha)=0 \ \text{ for all } \alpha\in \Sigma({\fg}_{\ol{0}}).
  \end{equation}

If $\dot{\fg}$ is of type II, then  $\mathbb{Z}\Delta=\mathbb{Z}\Sigma({\fg}_{\ol{0}})$, so~(\ref{rho'rho}) gives 
 $(\zeta,\beta)=0$ for all $\beta\in\Delta$.  This establishes the assertion for this case.

Consider the remaining case when $\dot{\fg}$ is of type I.  Recall that for $\dot{\fg}$ is of type I, (\ref{ika})
holds only for $\fsl(2|n)$.  Since $\mathbb{Z}\Delta$ is spanned by $\Sigma({\fg}_{\ol{0}})$ by $\vareps_1-\delta_1$, it is enough to check $(\zeta,\vareps_1-\delta_1)=0$.
In this case  $\vareps-1-\delta_1,\delta_n-\vareps_2\in\Sigma$, 
so $(\rho, \vareps_1-\delta_1+\vareps_2-\delta_n)=0$.
Note that  $(\theta,\vareps_1+\vareps_2)=0$. 
Since   $\dot{\rho}_-$ is
 the standard Weyl vector for $\fsl_n$ we have  $(\dot{\rho}_-,\delta_1+\delta_n)=0$. Therefore
 $$(\zeta, \vareps_1-\delta_1)+(\zeta, \vareps_2-\delta_n)=(\rho, \vareps_1-\delta_1+\vareps_2-\delta_n)=0.$$
Since $\vareps_1-\vareps_2$ and
$\delta_1-\delta_n$ lie in $\mathbb{Z}\Sigma({\fg}_{\ol{0}})$, we have
$(\zeta,\vareps_1-\delta_1)=(\zeta,\vareps_2-\delta_n)$. Thus 
 $(\zeta,\vareps_1-\delta_1)=0$ as required.
\end{proof}

\subsubsection{}
By~\cite{Kbook}, Theorem 4.3 
for a simple finite-dimensional Lie algebra $\ft$ we have
$$
\{\zeta\in\mathbb{Q}\Sigma(\ft)|\ \forall \alpha\in\Sigma(\ft)\ \  (\zeta,\alpha)\in\mathbb{Q}_{\geq 0}(\alpha,\alpha)\}\subset
\mathbb{Q}_{\geq 0}\Sigma(\ft).$$
This implies $\dot{\rho}_-\in \mathbb{Q}_{\geq 0}\Sigma(\dot\fg_{\ol{0}})_-$, so 
$(\eta, \dot{\rho}_-)\leq 0$ for any $\dot\fg_{\ol{0}}$-dominant weight $\eta\in\fh^*$.
Using~\Lem{lemrho-} we obtain 
\begin{equation}\label{lemeta}
(2\dot{\rho},\eta)\leq (h^{\vee}-1)(\eta,\theta) \ \text{ if } \eta\in\mathbb{Z}\Delta\ \text{ is $\dot\fg_{\ol{0}}$-dominant}.
\end{equation}

\subsection{Proof of~(\ref{check!})}\label{aaaa}
In this subsection $\Sigma$ satisfies~(\ref{ika}) and  the assumptions (i), (ii) of Section~\ref{Step4}.

In this case $(\beta,\theta)=\pm 1$ for all $\beta\in\dot{\Delta}_1$.
Let
 $\beta_1,\ldots,\beta_l\in \dot{\Delta}_1$ be the roots satisfying 
 $(\beta_i,\theta)=-1$. By above, $\dot{\Delta}=\{\pm\beta_i\}_{i=1}^l$ so
$$\Delta_{\ol{1}}^+\setminus\dot{\Delta}_1=\{j \delta\pm\beta_i|\ j\geq 1\}_{i=1}^l.
$$

For a fixed $A\in\cA$ we denote by $a_{i,\pm}$ the number of elements of the form
$j\delta\pm \beta_i$ in $A$; we set 
$$a_i:=a_{i,+}-a_{i,-}.$$
We have
$\mu_A=\sum\limits_{\scriptstyle{\alpha\in A}}\alpha= (\mu_A,\Lambda_0)\delta+ \sum\limits_{i=1}^l 
a_i\beta_i$.

\subsubsection{}
\begin{lem}{lemD2mu}
If $\dot{\fg}$ is as in Section~\ref{Step4},  then  $2 ||\mu_A||^2\leq l \sum\limits_{i=1}^l a_i^2$.
\end{lem}
\begin{proof}
Consider the case when $\dot{\fg}$ is of type II.  By~\Lem{mutheta}, 
$2||\mu_A||^2\leq (\mu_A,\theta)^2$.
Since $(\beta_i,\theta)=-1$ for all $i$, we have 
$(\mu_A,\theta)=-\sum\limits_{i=1}^l a_i$. Thus
$$2||\mu_A||^2\leq (\sum\limits_{i=1}^l a_i)^2.$$
By Jensen's inequality,  $(\sum\limits_{i=1}^l a_i)^2\leq l \sum\limits_{i=1}^l a_i^2$.
This completes the proof for  $\dot{\fg}$  of type II.

The remaining case is $\dot{\fg}=\fsl(2|m)$. In this case $l=2m$ and we have $\beta_i=\vareps_2-\delta_i$ and
$\beta_{m+i}=\delta_i-\vareps_1$ for $i=1,\ldots,m$. Then
$(\beta_i,\beta_j)=1$ if $1\leq i,j\leq m$ or
$m<i,j\leq2m$ of $j-i=\pm m$; in all other cases $(\beta_i,\beta_j)=0$. 
One has
$$||\mu_A||^2=2\sum\limits_{1\leq i<j\leq 2m} a_i a_j (\beta_i,\beta_j)$$
that is
$$m\sum\limits_{i=1}^l a_i^2-||\mu_A||^2=
\sum\limits_{1\leq i<j\leq m} (a_i-a_j)^2+\sum\limits_{m<i<j\leq 2m} (a_i-a_j)^2+\sum\limits_{i=1}^m (a_i-a_{m+i})^2\geq 0.$$
Since $l=2m$ we get $l\sum\limits_{i=1}^l a_i^2\geq 2 ||\mu_A||^2$ as required.
 \end{proof}

\subsubsection{}
\begin{lem}{lemDDD}
We have $\ 2(\mu_A,\Lambda_0)\geq \sum\limits_{i=1}^l (a_i^2+|a_i|)$.
\end{lem}
\begin{proof}
Recall that $A$ contains  the elements
$$j_1\delta+\beta_1, j_2\delta+\beta_1,\ldots, j_{a_{1,+}}\delta+\beta_1\ 
\ \text{ for }\ \ 1\leq j_1<j_2<\ldots <j_{a_{1,+}}.$$
The sum of these elements is  $m_{1,+}\delta+a_{1,+}\beta_1$
where 
$$m_{1,+}=j_1+\ldots+ j_{a_{1,+}}\geq 1+2+\ldots a_{1,+}=\frac{a_{1+}(a_{1+}+1)}{2}.$$ 
Using the similar notation for other elements of $\dot{\Delta}$
we obtain  
$$2(\mu_A,\Lambda_0)=2\sum\limits_{i=1}^l(m_{i,+}+m_{i,-})
\geq \sum\limits_{i=1}^l a_{i,+}(a_{i,+}+1)+a_{i,-}(a_{i,-}+1).$$
Since $a_{i,\pm}\in\mathbb{Z}_{\geq 0}$ we have
$a_{i,+}^2+a_{i,-}^2\geq (a_{i,+}-a_{i,-})^2=a_i^2$.
Hence 
$$2(\mu_A,\Lambda_0)\geq \sum\limits_{i=1}^l (a_i^2+|a_i|)$$
as required.
\end{proof}

\subsubsection{}
\begin{cor}{cor7}
Formula~(\ref{check!}) holds if   $k\in\mathbb{Z}_{\geq 0}$ is such that 
$k+2h^{\vee}\geq \frac{5}{2}$.
\end{cor}
\begin{proof}
Recall that~(\ref{check!}) asserts that for any $A\in\cA_+$ with $\mu_A\not\in\mathbb{Z}_{\geq 0}\delta$ the number
$$u_A:=2(k+h^{\vee})(\Lambda_0,\mu_A)+(1-h^{\vee})(-\mu_A,\theta)
-||\mu_A||^2$$
satisfies the inequalities $0<u_A$ and $ 2(2-h^{\vee}) (-\mu_A,\theta)\leq u_A$.

Retain notation of Section~\ref{aaaa} and recall that $\mu_A=(\mu_A,\Lambda_0)\delta+ \sum\limits_{i=1}^l 
a_i\beta_i$. Set
$$m_A:= (\mu_A,\Lambda_0),\ \ \ 
D:=\sum\limits_{i=1}^l a_i, \ \ D_+:=\sum\limits_{i=1}^l a_i,\ \ 
D_2:=\sum\limits_{i=1}^l a_i^2.$$
 Since $(\beta_i,\theta)=-1$ for all $i$, we have
  $(\mu_A,\theta)=-D$. Therefore 
$$u_A= 2(k+h^{\vee})m_A+(1-h^{\vee})D -||\mu_A||^2.$$

For $A\in\cA_+$ we have $(-\mu_A,\theta)\geq 0$, so 
$D\geq 0$.
Observe that  $\mu_A\in\mathbb{Z}_{\geq 0}\delta$ if  $a_i=0$ for all $i$; in particular,
$\mu_A\not\in\mathbb{Z}_{\geq 0}\delta$ implies 
$D_+\geq 1$. Therefore for~(\ref{check!}) it is enough to verify
 \begin{equation}\label{assmtool1}
 \begin{array}{lcll}
 D\geq 0,\ D_+>0\ & \Longrightarrow\ &  0<u_A,\ & (4-2h^{\vee}) D\leq u_A,\end{array}\end{equation} 
 
Indeed, assume that $D\geq 0$ and $D_+>0$.
By Lemmas~\ref{lemD2mu}, \ref{lemDDD} we have
$2m_A\geq D_2+D_+$ and $ l D_2\geq 2||\mu_A||^2$, so
$$u_A\geq (k+h^{\vee})(D_2+D_+)-\frac{lD_2}{2}+(1-h^{\vee})D.$$
By~(\ref{strange}) we have  $l+2h^{\vee}=4$ and $h^{\vee}\leq 1$. For $D\geq 0$ we get 
$$u_A\geq  (k+2h^{\vee}-2)D_2+(k+h^{\vee})D_++ (1-h^{\vee})D.$$

 By above,  $u_A>0$ if   $k+2h^{\vee}\geq 2$ and  $k+h^{\vee}>0$.
It remains to verify $ (4-2h^{\vee}) D\leq u_A$ that is
$$(4-2h^{\vee}) D\leq (k+2h^{\vee}-2)D_2+(k+h^{\vee})D_++ (1-h^{\vee})D$$
that is $$(3-h^{\vee}) D\leq (k+2h^{\vee}-2)D_2+(k+h^{\vee})D_+.$$

Since $D_2\geq D_+\geq D$, the above inequlaity holds if
$3-h^{\vee}\geq k+2h^{\vee}-2+ k+h^{\vee}$. Thus~(\ref{assmtool1}) holds for $k+2h^{\vee}\geq \frac{5}{2}$.
\end{proof}

\section{Proof of~\Thm{thm3} (vi), (vii)}\label{sect4}

In this section we show that $\fg$-module $V^k$ has length greater than two in the following  cases:
$k=-h^{\vee}$;   for all $\dot{\fg}=\fosp(4|2n)$ if $k\geq n$; and for
$\dot{\fg}=D(2|1,a)$ with  $a\in\mathbb{R}_{<-1}$ if $k\in(\mathbb{Z}_{\geq 0}\cap a^{-1}\mathbb{Z})$.

For $\dot{\fg}=D(2|1,a)$ we have   the following square lengths of positive even roots:
$||2\vareps_1||^2=2$, $||2\vareps_2||^2=-\frac{2}{a+1}$, $||2\vareps_3||^2=-\frac{2a}{a+1}$,
and  $\theta=2\vareps_1$.   
If  $||2\vareps_i||^2\not\in\mathbb{R}_{>0}$ for $i=2,3$, then  $\dot{\Delta}^{\#}=\{\pm\theta\}$, so
$\fs$ is of  unitary type . If $\fs$ is of non-unitary type, then $\dot{\Delta}^{\#}=\{\pm\theta, \pm\theta'\}$,
and without loss of generality, we may assume that
$\theta'= 2\vareps_2$, that is $a<-1$.

\subsection{Case $k=-h^{\vee}$}\label{critical}
Let us show that $V^{-h^{\vee}}$ is of infinite length. 

Consider the Sugawara operator $T_{-p}$ (see~\cite{Kbook}, (12.8.4)). Recall that $T_{-p}$ 
gives  an $[\fg,\fg]$-endomorphism of $V^{-h^{\vee}}$.
Set $\fg_{--}:=\sum\limits_{i=1}^{\infty} \dot{\fg} t^{-i}$.
Observe that  $V^k$ is a free $\cU(\fg_{--})$-module generated by the vacuum vector $v_0$.
Let $\{u'_i\}_{i=1}^q$ be an orthonormal basis of $\dot{\fg}$ (i.e., $(u'_i,u'_j)=\delta_{ij}$).
Using~\cite{Kbook}, (12.8.2), we obtain
$$T_{-2p-1}v_0=2\bigl(\sum\limits_{j=1}^{p}\sum\limits_{i=1}^q (u_i t^{-(2p+1-j)}) (u_i t^{-j})\bigr)v_0.$$
The PBW Theorem  for $\cU(\fg_{--})$ gives $T_{-2p-1}v_0\not=0$.
Therefore $T_{-2p-1}v_0$  are primitive vectors in $V^{-h^{\vee}}$,
so this module has subquotients $L(-h^{\vee}\Lambda_0-(2p+1)\delta)$ for all $p\in\mathbb{Z}_{\geq 0}$.

\subsection{}
\begin{prop}{propD21a}
The $\fg$-module $V^k$ has length greater than two in the following cases

\begin{enumerate}
\item $\dot{\fg}=\mathfrak{osp}(4|2n)$ if $k\geq n$;
\item $\dot{\fg}=D(2|1,a)$ with 
$a<-1$, if $k\in (\mathbb{Z}_{> 0}\cap a^{-1}\mathbb{Z})$;
\end{enumerate}
\end{prop}
\begin{proof}
In these cases  $\dot{\Delta}^{\#}=\{\pm \theta,\pm\theta'\}$
 is of type $D_2=A_1\coprod A_1$; we set $\alpha_0:=\delta-\theta$, $\alpha_0':=\delta-\theta'$.

For (i) we have $\dot{\fg}=\mathfrak{osp}(4|2n)$ and 
$\theta=\vareps_1-\vareps_2$,
$\theta':=\vareps_1+\vareps_2$. We  take
$$\begin{array}{ll}
\dot{\Sigma}:=\{\vareps_1+\vareps_2, -\vareps_2-\delta_1,
\delta_1-\delta_2,\ldots,\delta_{n-1}-\delta_n,2\delta_n\},\  & \Sigma=\{\alpha_0\}\cup\dot{\Sigma}, \\
\dot{\Sigma}':=\{\vareps_1-\vareps_2, \vareps_2-\delta_1,
\delta_1-\delta_2,\ldots,\delta_{n-1}-\delta_n,2\delta_n\},\  & \Sigma'=\{\alpha'_0\}\cup\dot{\Sigma}', \end{array}$$
and we denote by $L_{\Sigma'}(\lambda)$ the simple $\fg$-module with the highest weight $\lambda$
with respect to $\Sigma'$.

Recall that if a base $\Sigma$ is obtained from a base $\Sigma''$
by an odd reflection with respect to $\beta\in\Sigma''$, 
then $L(\lambda)=L_{\Sigma''}(\lambda'')$,
where $\lambda=\lambda''$ if $(\lambda'',\beta)=0$, and 
$\lambda=\lambda''-\beta$ otherwise. 
The base ${\Sigma}$ is obtained from the base $\Sigma'$ 
by the sequence of odd reflections with respect to the roots
$$\vareps_2-\delta_1,\vareps_2-\delta_2,\ldots, \vareps_2-\delta_n,\vareps_2+\delta_n,\vareps_2+\delta_{n-1},\ldots, \vareps_2+\delta_1.$$
Therefore $L_{\Sigma'}(\lambda')=L(\lambda)$ implies that 
$$\lambda'-\lambda=\sum\limits_{i=1}^n p_i (\vareps_2-\delta_i)+q_i (\vareps_2+\delta_i)\ \ \ \text{ for some } p_i,q_i\in\{0,1\}.$$
Observe that $\lambda'-\lambda$ is a linear combination of
$\vareps_2$ and $\delta_1,\ldots,\delta_n$, with the coefficient of $\vareps_2$ less than or equal to $2n$.
One has $[V^k:L(s_0.k\Lambda_0)]=1$. 
Since $\frac{2(k\Lambda_0,\alpha'_0)}{(\alpha'_0,\alpha'_0)}=k$ 
 we have 
$$[V^k:L_{\Sigma'}(\lambda')]=1,\ \ \text{ where }
\lambda':=k\Lambda_0-(k+1)\alpha_0'.$$
If the $\fg$-module  $V^k$ has length two, then 
$L(s_0.k\Lambda_0)=L_{\Sigma'}(\lambda')$. One has
$$\lambda'-s_0.k\Lambda_0=(k+1)\alpha_0-(k+1)\alpha_0'=2(k+1)\vareps_2$$
and the above observation forces $2(k+1)\leq 2n$ that is $k\leq n-1$.
This establishes (i).

The proof of (ii) is similar. We have $\dot{\fg}:=D(2|1;a)$ and
$\dot{\Delta}^{\#}=\{\pm \theta,\pm\theta'\}$, where
$\theta=2\vareps_1$,
$\theta':=2\vareps_2$.
We take
$$\dot{\Sigma}:=\{2\vareps_2,\vareps_1-\vareps_2-\vareps_3,2\vareps_3\},\ \ 
\dot{\Sigma}':=\{2\vareps_1,-\vareps_1+\vareps_2-\vareps_3,2\vareps_3\}.$$
Then  $\Sigma=\{\alpha_0\}\cup\dot{\Sigma}$ and 
$\Sigma':=\{\alpha'_0\}\cup\dot{\Sigma}$.

The base ${\Sigma}$ is obtained from the base $\Sigma'$ 
by the odd reflections with respect to the roots 
$-\vareps_1+\vareps_2-\vareps_3$ and $-\vareps_1+\vareps_2+\vareps_3$.
Therefore $L_{\Sigma'}(\lambda')=L(\lambda)$ implies that 
\begin{equation}\label{D21apq}
\lambda'-\lambda=p(-\vareps_1+\vareps_2-\vareps_3)+q(-\vareps_1+\vareps_2+\vareps_3)\ \ \ \text{ for some } p,q\in\{0,1\}.
\end{equation}

For $k\in (\mathbb{Z}_{>0}\cap a^{-1}\mathbb{Z})$ we have
 $[V^k:L(s_0.k\Lambda_0)]=1$. Since
$\frac{2(k\Lambda_0,\alpha'_0)}{(\alpha'_0,\alpha'_0)}=-(a+1)k\in\mathbb{Z}_{\geq 0}$ 
 we have 
$$[V^k:L_{\Sigma'}(\lambda')]=1,\ \ \text{ where }
\lambda':=k\Lambda_0-(1-(a+1)k)\alpha_0'.$$
Assume that  the $\fg$-module $V^k$ has length two. Then
$L(s_0.k\Lambda_0)=L_{\Sigma'}(\lambda')$ and~(\ref{D21apq}) gives
$$\lambda'-s_0.k\Lambda_0=-(p+q)\vareps_1+(p+q)\vareps_2+(q-p)\vareps_3 
\ \ \ \text{ for some } p,q\in\{0,1\}.$$
Since
$$\lambda'-s_0.k\Lambda_0=(k+1)\alpha_0-(1-(a+1)k)\alpha_0'=(2+a)k\delta-
2(k+1)\vareps_1+2(1-(a+1)k)\vareps_2$$
this is impossible for $k>0$. This establishes (ii).
\end{proof}

\section{Proof of~\Thm{thm3}  (i)--(v)}\label{jantzen}
In this section we may assume that 
$h^{\vee}\geq 0$ and  $k\in\mathbb{Z}_{\geq 0}$ is such that
$k+h^{\vee}\not=0$. We also assume that $\dot{\fg}\not=D(2|1,a)$, since this case was taken care of in 
Section~\ref{sect4}.
Our goal is to get a necessary condition on $k$ for the $\fg$-module
 $V^k$ being of  length two.

Our main tool is the Jantzen filtration. Recall that 
the Jantzen filtration of $V^k$ is a decreasing filtration
$\{\cJ^i(V^k)\}_{i=0}^{\infty}$   with $\cJ^0(V^k)=V^k$ and $\cJ^1(V^k)$ being the maximal proper
submodule of $V^k$.  
The Jantzen sum formula expresses $\sum\limits_{i=1}^{\infty} \ch \cJ^i(V^k)$ in terms of the exponents
of the vacuum Shapovalov determinants, computed in~\cite{GKvac}.  Hence $\cJ^2(V^k)=0$ is equivalent to the fact that 
this sum equals  to $\ch  \cJ^1(V^k)$. Note that 
 $\ch  \cJ^1(V^k)=\ch V^k-\ch V_k$
where $V_k$ is the simple quotient of $V^k$.
For $k\in\mathbb{Z}_{\geq 0}$,  $\ch V_k$ is given by the Weyl-Kac character formula
if $\fg$ is a Lie algebra and by
 the Kac-Wakimoto  formula in other cases. Comparing the  Jantzen sum formula
 and $\ch V^k-\ch V_k$ we obtain a criterion when 
the Jantzen filtration has length two (i.e. $\cJ^2(V^k)=0$). 

It is not hard to see that if $V^k$ has length two, then $\cJ^2(V^k)=0$ (see~\Cor{cors=11}).
In this way, we  get a necessary condition for $V^k$ being of  length two.

\subsection{Notation}
We set $Q^{\#}:=\mathbb{Z}\Delta^{\#}$.
 We denote by $\dot{W}$ ($\subset W$) the Weyl group
of  $\dot{\Delta}_0$. For each subset $A\subset \Delta_{\ol{0}}$ we denote  by $W(A)$ the subgroup generated
by the reflections with respect to the roots in $A$. For any subset $B\subset W$ we define the linear operator
$$\cF_{B}:=\sum\limits_{w\in B} \sgn(w) w.$$

We set
$$\supp\bigl(\sum\limits_{\nu} a_{\nu} e^{\nu}\bigr):=\{\nu|\ a_{\nu}\not=0\}.
$$

We fix the triangular decomposition in such a way that $\dot{\Sigma}$ contains a maximal isotropic set $S$
($S\subset \dot{\Delta}$ is isotropic if and only if   $(S,S)=0$).

We intoduce $\dot{R},\dot{K}(\nu)$ by the formulae
$$\dot{R}:=\prod\limits_{\alpha\in\dot{\Delta}_{\ol{0}}^+} (1-e^{-\alpha})^{\dim\fg_{\alpha}}
\prod\limits_{\alpha\in\dot{\Delta}_{\ol{1}}^+} (1+e^{-\alpha})^{-\dim\fg_{\alpha}}
=:\sum\limits_{\mu\in \dot{Q}^+} \dot{K}(\mu)e^{-\mu}.$$
and set
$$
R:=\prod\limits_{\alpha\in\Delta_{\ol{0}}^+} (1-e^{-\alpha})^{\dim\fg_{\alpha}}
\prod\limits_{\alpha\in\Delta_{\ol{1}}^+} (1+e^{-\alpha})^{-\dim\fg_{\alpha}}.$$

For each $k\in\mathbb{C}$ we set
$$U(k):=\{ \eta\in Q^+|\ ||k\Lambda_0+\rho-\eta||^2=||k\Lambda_0+\rho||^2\}.$$

\subsection{Denominator Identity}\label{sectdenomfin}
For $\dot{\fg}\not=\osp(4|2n)$,  $\mathfrak{spo}(2n|4)$, $D(2|1,a)$,
$[\dot{\fg}_{\ol{0}},\dot{\fg}_{\ol{0}}]$ is a product of two Lie algebras,
where the first one is simple 
and the second one is either simple or zero. Hence the
 finite Weyl group $\dot{W}$ can be decomposed as $\dot{W}=\dot{W}'\times \dot{W}''$ where $\dot{W}'$, $\dot{W}''$ are the Weyl groups
of the simple components of $\dot{\fg}_{\ol{0}}$ and $|\dot{W}'|\geq |\dot{W}''|$ (if
$[\dot{\fg}_{\ol{0}},\dot{\fg}_{\ol{0}}]$ is simple, then $\dot{W}''=\{\Id\}$).
For $\dot{\fg}=\osp(4|2n)$,  $\mathfrak{spo}(2n|4)$ with $n\geq 2$, we have   $\dot{W}=\dot{W}'\times \dot{W}''$,
where $\dot{W}'$ is the Weyl group of $\mathfrak{sp}_{2n}$ and  $\dot{W}''\cong \mathbb{Z}_2\times \mathbb{Z}_2$
 is the Weyl group of $\mathfrak{o}_{4}$.

 The denominator identity for $\dot{\fg}$, established in~\cite{KWnum}, \cite{Gfin} , is 
$$
\dot{R}e^{\dot{\rho}}=j_0^{-1}\cF_{ \dot{W} }\bigl(
\frac{e^{\dot{\rho}}}{\prod\limits_{\beta\in S}(1+e^{-\beta})}\bigr)=\cF_{ \dot{W}'}\bigl(
\frac{e^{\dot{\rho}}}{\prod\limits_{\beta\in S}(1+e^{-\beta})}\bigr)$$
where $j_0:=|\dot{W}''|$, and $S\subset \dot{\Delta}^+$ is a maximal isotropic subset.
Recall that in this section $h^{\vee}\geq 0$.
If $h^{\vee}>0$, then  $\dot{W}'=\dot{W}^{\#}$. If 
$h^{\vee}=0$ and $\dot{\fg}=\fgl(n|n)$, $\mathfrak{psl}(n|n)$, then $\dot{W}'\cong \dot{W}''$ and we 
 can (and will) choose $\dot{W}':=\dot{W}^{\#}$. Thus, for
 $\dot{\fg}\not=\mathfrak{spo}(2n+2|2n)$, $D(2|1,a)$, we have $\dot{W}'=\dot{W}^{\#}$.
 This gives the following form of
the  denominator identity: 
$$\dot{R}e^{\dot{\rho}}=\cF_{ \dot{W}^{\#} }\bigl(
\frac{e^{\dot{\rho}}}{\prod\limits_{\beta\in S}(1+e^{-\beta})}\bigr).$$

 \subsubsection{}\label{expansionrule}
For $w\in W$ and $\mu=\sum\limits_{\beta\in S} m_{\beta}\beta$ with $m_{\beta}\geq 0$
we set
\begin{equation}\label{wmudef}\begin{array}{l}
|w\mu|=\sum\limits_{\beta\in S: w\beta\in\Delta^+} m_{\beta} w\beta-
\sum\limits_{\beta\in S: w\beta\not\in\Delta^+} (m_{\beta}+1) w\beta.
\end{array}
\end{equation}
(Note that our notation differs from the notation in~\cite{HRvac}).  One has
\begin{equation}\label{wmuprop}\begin{array}{lcl}
(i) & & |w\mu|\in  Q^+\cap (w\mathbb{Z}S),\\
(ii) & & |w\mu|=|w\mu'|\ \ \Longrightarrow\ \ \mu=\mu',\\
(iii) & & |w\mu|\in Q^{\#}\ \Longrightarrow |w\mu|=0 \ \Longrightarrow\  \mu=0, \ wS\subset\Delta^+.
\end{array}\end{equation}
The  formulas (i), (ii) 
and the last implication of (iii) immediately follow from the definition; for the first implication of (iii)
note that $Q^{\#}\cap (w\mathbb{Z}S)=w(Q^{\#}\cap \mathbb{Z}S)=0$, so 
$|w\mu|\in Q^{\#}$
implies $|w\mu|=0$ as required.

Let  $\htt\mu:=\sum\limits_{\beta\in S} m_{\beta}$. We have
$$w(\prod\limits_{\beta\in S}(1+e^{-\beta})^{-1})=\sum\limits_{\mu\in\mathbb{Z}_{\geq 0}S}
(-1)^{\htt \mu}e^{-|w\mu|}=\sum\limits_{\mu\in\mathbb{Z}_{\geq 0}S}(-1)^{p(\mu)}e^{-|w\mu|}.$$

\subsubsection{}
Now we can rewrite the denominator identity 
as
\begin{equation}\label{denomexp}
\sum\limits_{\nu\in \dot{Q}} \dot{K}(\nu)e^{\dot{\rho}-\nu}=  \dot{R}e^{\dot{\rho}}=\sum\limits_{w\in\dot{W}'}\sum\limits_{\mu\in\mathbb{Z}_{\geq 0}S} (-1)^{p(\mu)}\sgn(w) e^{w\dot{\rho}-|w\mu|}.
\end{equation}

\subsubsection{}
Take any $k'\in\mathbb{C}$.
Since $\rho+k'\Lambda_0-\dot{\rho}=(k'+h^{\vee})\Lambda_0$ is $\dot{W}$-invariant, the denominator identity gives
\begin{equation}\label{denid2}
\dot{R}e^{k'\Lambda_0+\rho}=j_0^{-1}\cF_{ \dot{W} }\bigl(
\frac{e^{k'\Lambda_0+\rho}}{\prod\limits_{\beta\in S}(1+e^{-\beta})}\bigr)
=\cF_{ \dot{W}' }\bigl(
\frac{e^{k'\Lambda_0+\rho}}{\prod\limits_{\beta\in S}(1+e^{-\beta})}\bigr)
\ \text{ for any } \ 
k'\in\mathbb{C}.\end{equation}
Take $\nu\in -\supp\dot{R}$.  By~(\ref{denomexp})  we have $\dot{\rho}-\nu=w\dot{\rho}-|w\mu'|$
for some $w\in\dot{W}'$  and $\mu'\in\mathbb{Z}_{\geq 0} S$; by~(\ref{wmuprop}) (i) we have $|w\mu'|=w\mu$
for some $\mu\in \mathbb{Z}S$, so
$\dot{\rho}-\nu=w(\dot{\rho}+\mu)$. Therefore
$$k'\Lambda_0+\rho-\nu=(k'+h^{\vee})\Lambda_0+w(\dot{\rho}+\mu)=w(k'\Lambda_0+\rho+\mu).$$
Clearly, $ \supp\dot{R}\subset \dot{Q}^+$. We obatin
\begin{equation}\label{dotk} 
-\supp \dot{R}\subset \{\alpha\in \dot{Q}^+|\  ||k'\Lambda_0+\rho-\alpha||^2=||k'\Lambda_0+\rho||^2\}.
\end{equation}

Alternatively, (\ref{dotk}) can be obtained by using the Casimir operator on the Verma module $M(k'\Lambda_0)$.

\subsubsection{}
\begin{lem}{}
If $\dot{\fg}\not=\mathfrak{spo}(2n|2n+2)$ and 
 $(\alpha,\alpha)\geq 0$ for all $\alpha\in\dot{\Sigma}$, then
\begin{equation}\label{denomsgny}
\sum\limits_{y\in \Stab_{\dot{W}^{\#}}\dot{\rho}:\ yS\subset\dot{\Delta^+}} \sgn(y)=1.\end{equation}
\end{lem}
\begin{proof}
The assumption $(\alpha,\alpha)\geq 0$ for all $\alpha\in\dot{\Sigma}$
gives $(\dot{\rho},\alpha)\geq 0$ for all $\alpha\in\dot{\Delta}^+$ (since $(\dot{\rho},\alpha_i)=\frac{1}{2}
(\alpha_i,\alpha_i)$ for $\alpha_i\in\dot{\Sigma}$).
In particular,  for $w\in\dot{W}^{\#}$ we have $\dot{\rho}-w\dot{\rho} \in Q^+$, so  the coefficient of 
 $e^{\dot{\rho}}$ in the right-hand side of~(\ref{denomexp}) 
equals to 
$$\sum\limits_{y\in \Stab_{\dot{W}^{\#}}\dot{\rho}:\ yS\subset\dot{\Delta^+}} \sgn(y).$$
Since the coefficient of $e^{\dot{\rho}}$ in $\dot{R}e^{\dot{\rho}}$ equals to $1$, this gives
the required formula.
\end{proof}

\subsection{Jantzen sum formula}
For $k\in\mathbb{C}$ we 
have the following formula, up to a constant factor,  for the vacuum Shapovalov determinant
(see formula (11) and 3.2.3 in~\cite{GKvac})
\begin{equation}\label{cordet}
\begin{array}{l}
\det S_{\nu}(k)=\prod\limits_{r=1}^{\infty}
\prod\limits_{\gamma\in {\Delta}^+\setminus \dot{\Delta}}
\prod\limits_{\alpha\in \dot{Q}^+}\phi_{r\gamma+\alpha}(k)^{d_{r,\gamma,\alpha}(\nu)},\  \text{ where }\\
R\sum\limits_{\nu} d_{r,\gamma,\alpha}(\nu) e^{-\nu}=(-1)^{(r+1)p(\gamma)}(\dim \fg_{\gamma})\, \dot{K}(\alpha) e^{-\alpha-r\gamma},
\end{array}
\end{equation}
and, for $\mu\in \fh^*$, we set
$$\phi_{\mu}(k):=2(k\Lambda_0+\rho,\mu)-(\mu,\mu)=||k\Lambda_0+\rho||^2-||k\Lambda_0+\rho-\mu||^2.$$

Recall that $\phi_{\mu}(k)=0$ if and only if $\mu\in U(k)$.
The determinant formula~(\ref{cordet}) can be rewritten as
\begin{equation}\label{dmu}\begin{array}{l}
\det S_{\nu}(k)=\prod\limits_{\mu\in U(k)} \phi_{\mu}(k)^{d_{\mu}(\nu)} ,\ \ 
\text{ where }\ 
d_{\mu}(\nu):=\sum\limits_{r=1}^{\infty}
\sum\limits_{\gamma\in {\Delta}^+\setminus \dot{\Delta}} d_{r,\gamma,\mu-r\gamma}(\nu),\ \text{ and }\\
R\sum\limits_{\nu\in Q} d_{\mu}(\nu)e^{-\nu}=\sum\limits_{r=1}^{\infty}
\sum\limits_{\gamma\in {\Delta}^+\setminus \dot{\Delta}}
(-1)^{(r+1)p(\gamma)}(\dim \fg_{\gamma} )\,\dot{K}(\mu-r\gamma) e^{-\mu}.\end{array}\end{equation}

Note that for each $\nu\in Q^+$ one has
$d_{\mu}(\nu)\not=0$ for finitely many $\mu\in U(k)$.

\subsubsection{Jantzen filtration}
Recall that for $k\in\mathbb{C}$
the Jantzen filtration of $V^k$ is a decreasing filtration
$\{\cJ^i(V^k)\}_{i=0}^{\infty}$   with $\cJ^0(V^k)=V^k$ and $\cJ^1(V^k)$ being the maximal proper
submodule of $V^k$.  

The Jantzen sum formula~\cite{Jan} is
$$\sum\limits_{i=1}^{\infty} \ch \cJ^i(k)=\sum\limits_{\mu\in U(k)} d_{\mu}(\nu) e^{k\Lambda_0-\nu}.$$

In particular, for each weight space $\nu$ one has $\cJ^s(V^k)_{\nu}=0$ for $s>>0$.

Using~(\ref{dmu}) we get
\begin{equation}\label{Janmu}
Re^{\rho}\sum\limits_{i=1}^{\infty} \ch \cJ^i(k)=\sum\limits_{\mu\in U(k)} 
\sum\limits_{r=1}^{\infty}
\sum\limits_{\gamma\in {\Delta}^+\setminus \dot{\Delta}}
(-1)^{(r+1)p(\gamma)}\dim \fg_{\gamma} \,\dot{K}(\mu-r\gamma) e^{k\Lambda_0+\rho-\mu}.\end{equation}

\subsubsection{}
\begin{lem}{s=1}
For $k\in\mathbb{Z}_{\geq 0}$ we have
$\dim \cJ^1(V^k)_{s_0.k\Lambda_0}=1$ and  $\cJ^s(V^k)_{s_0.k\Lambda_0}=0$ for $s>1$.
\end{lem}
\begin{proof}
Recall that the Shapovalov form on $V^k$ up to a contstant factor does not depend on the choice of the 
base $\Sigma$; therefore 
the Jantzen filtration also does not depend on this choice.
Choose $\Sigma'$ containing $\alpha_0$.  Let $\mu\in U(k)$
be such that $d_{\mu}((k+1)\alpha_0)\not=0$. Then  
$$(\mu-(k+1)\alpha_0)\in\mathbb{Z}_{\geq 0}\Sigma',\ \ \  
2(k\Lambda_0+\rho,\mu)=(\mu,\mu).$$
Since $\alpha_0\in\Sigma'$ and $\mu\in Q^+\setminus\dot{Q}^+$, 
the first formula implies  $\mu=i\alpha_0$ for $0<i\leq (k+1)$;  the
second formula gives $i=k+1$. The formula~(\ref{dmu}) gives
$$R\sum\limits_{\nu}d_{k\alpha_0}(\nu)e^{-\nu}=e^{-(k+1)\alpha_0}$$
($(k+1;\alpha_0;0)$ is the only suitable triple $(r;\gamma;\alpha)$ in~(\ref{dmu})).
The Jantzen sum formula gives 
$$\sum\limits_{i=1}^{\infty} \cJ^i(k)_{k\Lambda_0-(k+1)\alpha_0} =
 d_{k\alpha_0}((k+1)\alpha_0)=1.$$
This implies the statement.
\end{proof}

\subsubsection{}
\begin{cor}{cors=11}
If $k\in\mathbb{Z}_{\geq 0}$ is such that 
 $V^k$ has length two, then 
$\cJ^2(V^k)=0$ and
$\ 
Re^{\rho}\sum\limits_{i=1}^{\infty} \ch \cJ^i(k)=\dot{R}e^{k\Lambda_0+\rho}-Re^{\rho}\ch L(k\Lambda_0)$.
\end{cor}
\begin{proof}
Assume that $V^k$ has length two and $L$ is the socle of $V^k$.
Then there exists $s\geq 1$ such that
 $\cJ^i(V^k)=L$ for $i=1,\ldots,s$ and $\cJ^i(V^k)=0$ for $i>s$. \Lem{s=1} gives $s=1$. 
Hence
$$\sum\limits_{i=1}^{\infty} \ch \cJ^i(k)=\ch L=\ch V^k-\ch L(k\Lambda_0)$$
and the statement follows from the formula $\ch V^k=R^{-1}\dot{R} e^{k\Lambda_0}$.
\end{proof}

\subsection{Kac-Wakimoto character formula}
Recall that this Section  $h^{\vee}\geq 0$, $k\in\mathbb{Z}_{\geq 0}$ is such that $k\not=-h^{\vee}$,
and $\dot{\fg}\not=D(2|1,a)$.  
The group $W^{\#}$ is the subgroup of $W$ generated by the reflections $s_{\alpha}$ for
$\alpha\in\Delta$ such that
$(\alpha,\alpha)>0$.

\subsubsection{Group $T$}\label{TQ}
By~\cite{Kbook},   $W^{\#}=\dot{W}^{\#}\rtimes T$, where $T\subset W^{\#}$
is the so-called  ``translation group''. 
We have $T=\{t_{\mu}\}_{\mu\in Q'}$ 
where $Q'$ is the lattice spanned by $\frac{2\alpha}{(\alpha,\alpha)}$ for $\alpha\in\dot{\Sigma}^{\#}$  and 
$t_{\mu}\in GL(\fh^*)$ is given by the formula
\begin{equation}\label{tnu}
t_{\mu}(\lambda):=\lambda+(\lambda,\delta)\mu-((\lambda,\mu)+\frac{(\lambda,\delta)}{2}
(\mu,\mu))\delta. \end{equation}
Observe that  $Q'$ and $Q^{\#}$  are lattices of the same rank.

\subsubsection{}
In~\cite{GKadm} we proved 
the following Kac-Wakimoto type character formula in the cases
when  $h^{\vee}\geq 0$, 
$k\in\mathbb{Z}_{\geq 0}$, 
and $k+h^{\vee}\not=0$:
\begin{equation}\label{KWformula}
R e^{\rho} \ch L(k\Lambda_0)=j_0^{-1}\cF_{W^{\#}\dot{W}} \bigl(\frac{e^{k\Lambda_0+\rho}}
{\prod\limits_{\beta\in S} (1+e^{-\beta})}\bigr),\end{equation}
where $W^{\#}\dot{W}$ is the subgroup of $W$ generated by the subgroups $W^{\#}$ and $\dot{W}$ (and 
$j_0$ is a positive integer introduced  in  Section~\ref{sectdenomfin}).
Thus $W^{\#}\dot{W}=\dot{W}\rtimes T$. Let 
$W':=\dot{W}'T$ be the subgroup of $W$ 
generated by the subgroups $\dot{W}'$ and $T$, where  $\dot{W}'$ is as in~\ref{sectdenomfin}. One has 
$$W'=\left\{\begin{array}{ll}
{W}^{\#}=\dot{W}^{\#}\rtimes T & \text{ for } \dot{\fg}\not=\mathfrak{spo}(2n|2n+2)\\
\dot{W}'\times T & \text{ for } \dot{\fg}=\mathfrak{spo}(2n|2n+2).\end{array}
\right.$$

Combining~(\ref{denid2}) and~(\ref{KWformula}) we obtain
$$\begin{array}{l}
R e^{\rho} \ch L(k\Lambda_0)=j_0^{-1}\cF_{W^{\#}\dot{W}} \bigl(\frac{e^{k\Lambda_0+\rho}}
{\prod\limits_{\beta\in S} (1+e^{-\beta})}\bigr)=\cF_{T}\bigl(j_0^{-1}\cF_{\dot{W}} \bigl(\frac{e^{k\Lambda_0+\rho}}
{\prod\limits_{\beta\in S} (1+e^{-\beta})}\bigr)\bigr)=\\
\cF_{T}\bigl(\cF_{\dot{W}'} \bigl(\frac{e^{k\Lambda_0+\rho}}
{\prod\limits_{\beta\in S} (1+e^{-\beta})}\bigr)\bigr)
=\cF_{{W}'} \bigl(\frac{e^{k\Lambda_0+\rho}}
{\prod\limits_{\beta\in S} (1+e^{-\beta})}\bigr)).\end{array}$$
Hence we obtain the following formula which holds for all $k\in\mathbb{Z}_{\geq 0}$,
$k\not=-h^{\vee}$, 
if $\fg\not=D(2|1,a)$:
\begin{equation}\label{KWformulashort}
R e^{\rho} \ch L(k\Lambda_0)=\cF_{W'} \bigl(\frac{e^{k\Lambda_0+\rho}}
{\prod\limits_{\beta\in S} (1+e^{-\beta})}\bigr).\end{equation}
For  $\dot{\fg}\not=\mathfrak{spo}(2n|2n+2)$ one has
$W'=W^{\#}$ and~(\ref{KWformulashort}) is the usual Kac-Wakimoto formula
 proven in Section 4 of~\cite{GKadm} for $\dot{\fg}\not=\mathfrak{osp}(2n+2|2n)$
 and $\dot{\fg}\not=D(2|1,a)$.

\subsection{}
\begin{cor}{cors=1}
If $k\in\mathbb{Z}_{\geq 0}$ is such that 
 $V^k$ has length two, then 
\begin{equation}\label{mainformulalong}
Re^{\rho}\sum\limits_{i=1}^{\infty} \ch \cJ^i(k)=-\sum\limits_{\nu\in Q'\setminus\{0\}} 
t_{\nu}\bigl(\displaystyle\cF_{ \dot{W}'}
(\frac{e^{k\Lambda_0+\rho}}{\prod\limits_{\beta\in S}(1+e^{-\beta})})\bigr).
\end{equation}
\end{cor}
\begin{proof}
Combining~(\ref{KWformulashort})  with~(\ref{denid2}) we obtain
$$Re^{\rho}\ch L(k\Lambda_0)-\dot{R}e^{k\Lambda_0+\rho}=\displaystyle\cF_{ W'\setminus\dot{W}'}\bigl(
\frac{e^{k\Lambda_0+\rho}}{\prod\limits_{\beta\in S}(1+e^{-\beta})}\bigr)
=\sum\limits_{\nu\in Q'\setminus\{0\}} 
t_{\nu}\bigl(\displaystyle\cF_{ \dot{W}'}
(\frac{e^{k\Lambda_0+\rho}}{\prod\limits_{\beta\in S}(1+e^{-\beta})})\bigr).$$
Using~\Cor{cors=11} we obtain~(\ref{mainformulalong}).
\end{proof}

\subsection{Remark: Lie algebra case}
In this case the Jantzen sum formula can be simplified 
(we will not use this simplification); a similar simplification
can be done for  $\fosp(1|2n)$. Namely, the following holds.

\subsubsection{}
\begin{prop}{}
 If $\dot{\fg}$ is a Lie algebra, then for $k\in\mathbb{Z}_{\geq 0}$ we have
$$Re^{\rho}\sum\limits_{i=1}^{\infty} \ch \cJ^i(k)=\sum\limits_{w\in\dot{W}} 
\sum\limits_{\alpha\in {\Delta}^+_{\re}\setminus \dot{\Delta}}
\sgn(w) e^{ws_{\alpha}(k\Lambda_0+\rho)}.$$
\end{prop}
\begin{proof}
One has $\dot{R} e^{\rho}=\cF_{\dot{W}}(e^{\rho})$ 
and $\Stab_{\dot{W}}\rho=\{\Id\}$,
which gives
$$\dot{K}(\nu)=\left\{\begin{array}{ll}
\sgn(w) & \text{ if }\ \ \nu=\rho-w\rho\ \text{ for some } w\in \dot{W}\\
0 & \text{ otherwise}.
\end{array}\right.$$
Take any $k\not=-h^{\vee}$.
Assume that  $\mu\in U(k)$, $\gamma\in {\Delta}^+\setminus \dot{\Delta}$ and $r\geq 1$ are such that $\dot{K}(\mu-r\gamma)\not=0$. By above,
$$\mu=\rho-w\rho+r\gamma\ \text{ for some } w\in\dot{W},$$
so $\mu\in U(k)$ implies
$$||k\Lambda_0+\rho-(\rho-w\rho+r\gamma)||^2=||k\Lambda_0+\rho||^2,$$
that is  $||k\Lambda_0+w\rho-r\gamma||^2=||k\Lambda_0+\rho||^2$ or, equivalently,
$$||k\Lambda_0+\rho-r\alpha||^2=||k\Lambda_0+\rho||^2\ \text{ for }\alpha:=w^{-1}\gamma.$$
This implies $\alpha\not\in\mathbb{Z}\delta$ and
$k\Lambda_0+\rho-r(w^{-1}\gamma)=s_{\alpha}(k\Lambda_0+\rho)$.
Then 
$$k\Lambda_0+\rho-\mu=k\Lambda_0+w\rho+r\gamma=ws_{\alpha}(k\Lambda_)+\rho).$$
Using~(\ref{Janmu}) we obtain
$$Re^{\rho}\sum\limits_{i=1}^{\infty} \ch \cJ^i(k)=\sum\limits_{w\in\dot{W}}
\sum\limits_{\alpha\in A_k }
\sgn(w) e^{ws_{\alpha}(k\Lambda_0+\rho)}$$
where
$$A_k:=\{\alpha\in {\Delta}^+_{\re}\setminus \dot{\Delta}|\  k\Lambda_0-s_{\alpha}.(k\Lambda_0)\in\mathbb{Z}_{\geq 1}\alpha\}.$$
For   $k\in\mathbb{Z}_{\geq 0}$, the module $L(k\Lambda_0)$ is integrable; this implies
 $A_k={\Delta}^+_{\re}\setminus \dot{\Delta}$ as required.
\end{proof}

\subsubsection{}
Using $\Stab_W(k\Lambda_0+\rho)=\{\Id\}$, we obtain the following criterion:
if $\dot{\fg}$ is a Lie algebra and  $k\in\mathbb{Z}_{\geq 0}$, then the Jantzen filtration of $V^k$ has length $2$ if and only if
$$W\setminus\dot{W}=\{ws_{\alpha}|\ \alpha\in\Delta_{re}^+\setminus\dot{\Delta},\ \  w\in\dot{W}\}.$$
It is not hard to see that this conidtion holds only for $\dot{\fg}=\fsl_2$.

%

\subsection{More notation}\label{morenot}
In this subsection we introduce notation which will used in the proof of~\Thm{thm3}.

\subsubsection{Lattice cones}\label{cone}
For $\lambda\in\fh^*$ and linearly independent vectors $\alpha_1,\ldots,\alpha_s\in\fh^*$ we call
the set $\{\lambda+\sum\limits_{i=1}^s \mathbb{Z}_{\geq 0}\alpha_i\}\subset\fh^*$ 
an {\em  $s$-dimensional lattice cone}.

We will use the following observation: 

{\em a lattice cone of dimension $s$ does not lie in a 
finite union of  lattice cones of   smaller dimension.}

This follows from the fact that, if we identify $\fh^*$ with a real vector space $\mathbb{R}^{2\dim\fh^*}$, 
and denote by $N(R)$ the number of points  in the intersection of a  lattice cone of dimension $s$ and  the ball of raduis $R$, then $N(R)\sim CR^s$ for some $C>0$.

\subsubsection{Map $\dot{P}$}
We  define the projections $\dot{P}: Q\to \dot{Q}$  by the formula
$\dot{P}(\nu):=\dot{\nu}$, where $\nu=\dot{\nu}+a\delta$ for some $a\in\mathbb{C}$.
  Note that the restriction of $\dot{P}$ to $U_k$ is injective
(since $k+h^{\vee}\not=0$).  
 
 \subsubsection{Maps ${P}_+, P_S$}
Recall that  $\dot{\Sigma}$ contains a maximal isotropic set $S$: this is a set
of linearly independent mutually orthogonal roots  with the cardinality equal to the defect of $\dot{\fg}$.

Let $\dot{\fg}\not=\mathfrak{spo}(2n|2n+2)$. In this case
\begin{itemize}
\item  $\dot{\Delta}$ admits a base consisting of  roots with non-negative square lengths;
  \item
 the union $\dot{W}^{\#}(-S)\cup \dot{W}^{\#} S$ contain all isotropic roots in $\dot{\Delta}$.\end{itemize}
 As a consequence, $\dot{Q}=\mathbb{Z}S\oplus \dot{Q}_+$
 where $\dot{Q}_+$ is the span of $\{\alpha\in\dot{\Delta}|\ (\alpha,\alpha)>0\}$; notice that
 $\dot{Q}_+$ is the lattice which contains the lattice $\mathbb{Z}\dot{\Delta}^{\#}$, 
 these lattices are equal if $\dot{\Delta}$ does not have odd non-isotropic roots.
  We  define projections  $P_+: Q\to \dot{Q}_+$, $P_S: Q\to  
 \mathbb{Z}S$ 
  by the formulae
$ P_+(\nu):=\nu_+$,  $P_S(\nu):=\nu_S$,
where $\dot{\nu}=\dot{\nu}_++\nu_S$ for  $\dot{\nu}_+\in \dot{Q}_{+}$ and ${\nu}_S\in \mathbb{Z}S$.

\subsubsection{Sets $Y_{\cJ}(\mu)$, $Y_L(\mu)$}
We fix $k\in\mathbb{Z}_{\geq 0}$ such that $k+h^{\vee}\not=0$. 
We set
$$\begin{array}{lcl}
q_{\cJ}:=e^{-k\Lambda_0}R\sum\limits_{i=1}^\infty \ch \cJ^i(k), & & q_L:=e^{-k\Lambda_0-\rho}\sum\limits_{\nu\in Q'\setminus\{0\}} 
t_{\nu}\bigl(\displaystyle\cF_{ \dot{W}'}
(\frac{e^{k\Lambda_0+\rho}}{\prod\limits_{\beta\in S}(1+e^{-\beta})}).\end{array}$$
By~\Cor{cors=1},  we have
\begin{equation}\label{34}
\supp(q_{\cJ})=\supp(q_{L})\ \ \text{ if $V^k$ has length $2$}.
\end{equation}

For $\dot{\fg}\not=\mathfrak{spo}(2n|2n+2)$  for each $\mu\in\mathbb{Z}S$ we set
$$\begin{array}{l}
Y_{\cJ}(\mu):=\{\eta\in \dot{Q}_+|\ -(\eta+\mu)\in\dot{P}\bigl(\supp(q_{\cJ})\bigr)\}, \\
Y_{L}(\mu):=\{\eta\in \dot{Q}_+|\ -(\eta+\mu)\in\dot{P}\bigl(\supp(q_{L})\bigr)\}.\end{array}$$
By~(\ref{34}), $V^k$ has length greater than $2$ if  $Y_{\cJ}(\mu_0)\not=Y_L(\mu_0)$ for some 
$\mu_0$.

\subsection{ Proof of~\Thm{thm3} (i)}\label{defectzero}
Let us illustrate  the proof of~\Thm{thm3} on the case when $\dot{\fg}$ has zero defect
(another, more complicated proof, is given in~\cite{GKvac}).
In this case  $S$ is empty.  We fix $k\in\mathbb{Z}_{\geq 0}$.

One has $\dot{R}e^{\dot{\rho}}=\cF_{\dot{W}}(e^{\dot{\rho}})$, so
 $\supp(R)=\{w\dot{\rho}-\dot{\rho}\}_{w\in\dot{W}}$. Therefore
$$\dot{P}(\supp(q_{\cJ}))\subset \{w\dot{\rho}-\dot{\rho}+r\gamma|\ w\in\dot{W},\ \gamma\in\dot{\Delta}\cup\{0\}, r\geq 1\}.$$
Since $\dot{W}$ and $\dot{\Delta}\cup\{0\}$ are finite sets, the set $\dot{P}(\supp(q_{\cJ}))$ lies in a finite union of one-dimensional affine
lattice cones. On the other hand, we have
$$q_L=e^{-k\Lambda_0-\rho}\sum\limits_{\nu\in Q'\setminus\{0\}} 
t_{\nu}\bigl(\displaystyle\cF_{ \dot{W}^{\#}}(e^{k\Lambda_0+\rho})\bigr).$$
Since $k\geq 0$ we have 
$(k\Lambda_0+\rho,\alpha)>0$ for all $\alpha\in\Sigma$; this gives $\Stab_W(k\Lambda_0+\rho)=\{\Id\}$.
Using~(\ref{tnu}) we obtain
$$\dot{P}(\supp(q_L))=\{(k+h^{\vee})\alpha+w\dot{\rho}-{\dot{\rho}}|\ \alpha\in {Q}'\setminus\{0\},\ w\in\dot{W}\}.$$
In particular, $\dot{P}(\supp(q_L))\cup\{0\}$ contains $(k+h^{\vee}) {Q}'$. Since $Q'\cong \mathbb{Z}^m$, where $m$ is the rank of $\dot{\fg}$, the set $\dot{P}(\supp(q_L))\cup\{0\}$  contains a lattice cone of dimension $m$. (In fact, since $\dot{W}$ is finite, this set is a finite union of  lattice cone of dimension $m$). By above,  
$\dot{P}(\supp(q_{\cJ}))$ lies in a finite union of
lattice cones of dimension $1$.  Hence $\dot{P}(\supp(q_{\cJ}))\not=\dot{P}(\supp(q_L))$ for $m>1$,
in contradiction with~(\ref{34}). If $m=1$, then $\dot{\fg}=\fsl_2$ or
$\mathfrak{spo}(2|1)$, and the claim follows from~\Thm{thm2}.

\subsection{Plan of the proof of~\Thm{thm3} (ii)--(iv)}
From now on until Section~\ref{spo2n2n+2} we assume  $\dot{\fg}\not=\mathfrak{spo}(2n|2n+2)$.
By~\Cor{cors=1} it is sufficies  to verify that   $q_{\cJ}+q_L\not=0$.
By~\Cor{YJmu0}, for all $\mu_0\in\mathbb{Z}S$ the set $Y_{\cJ}(\mu_0)$ lie in a finite union of one-dimensional
affine lattice cones. Thus it  is enough to find $\mu_0$ such that
 $Y_L(\mu_0)\cup\{0\}$ contains a two-dimensional lattice cone.

We call a weight $\lambda\in\fh^*$ $W^{\#}$-{\em regular} if $\Stab_{W^{\#}}\lambda=\{\Id\}$.
Recall that $Q'\cong \mathbb{Z}^{\ell}$ where $\ell$ is the rank of $\dot{\fg}^{\#}$; one has
 $\ell>1$.  Using~\Cor{coreta} we show that
$Y_{L}(\mu_0)\cup\{0\}$ contains an $\ell$-dimensional lattice cone  if $\mu_0\in\mathbb{Z}S$ is such that 
$k\Lambda_0+\rho-\mu_0$ is $W^{\#}$-regular.   This approach generalizes the defect zero case:  $k\Lambda_0+\rho$ is $W^{\#}$-regular if $\deff\dot{\fg}=0$.
Note that
for an  integral  $W^{\#}$-regular weight $\lambda$ one has $(\lambda,\alpha^{\vee})\geq (\rho^{\#},\alpha^{\vee})$ for any $\alpha\in\Sigma(\fg^{\#})$, so $(\lambda,\delta)\geq (\rho^{\#},\delta)=h^{\vee}_{\#}$
where $h^{\vee}_{\#}$ is the dual Coxeter number of $\dot{\fg}^{\#}$.
Thus for $h^{\vee}\in\mathbb{Z}$ the $W^{\#}$-regularity of $k\Lambda_0+\rho-\mu_0$ implies
$k\geq h^{\vee}_{\#}-h^{\vee}$. In Section~\ref{thm3iii} we will exhibit such $\mu_0$
for each pair $(\dot{\fg},k)$ listed in~\Thm{thm3} (iv). 

For the rest of the cases we show that $Y_{L}(\mu_0)\cup\{0\}$ contains a two-dimensional lattice cone.
We check this in the following way. We fix $\Sigma$ containing a maximal isotropic set $S$ in such a way that
$(\alpha,\alpha)\geq 0$ for all $\alpha\in\Sigma$.  
Then $k\Lambda_0+\rho$ is $\fg^{\#}$-dominant and
$\Stab_{W^{\#}}(k\Lambda_0+\rho)=\Stab_{\dot{W}^{\#}} \dot{\rho}$ if $(k\Lambda_0+\rho,\alpha_0)>0$
(this holds for all $k\geq 0$ if $\alpha_0\in\Sigma$ and for $k>0$ otherwise).
We set 
$$\dot{W}_{\rho}:=\Stab_{\dot{W}^{\#}} \dot{\rho}.$$

By~\Cor{corbeta2}, $Y_L(0)\cup\{0\}$ contains $(k+h^{\vee})\eta$ if $\eta$ lies in the set 
$$\{\eta\in Q'|\ (\eta,\beta')\geq 0 \ \text{ for all } \beta'\in  (\dot{W}_{\rho}S\cap \Delta^+),\ \ (\eta,\beta')\geq 0 \ \text{ for all } \beta'\in  (\dot{W}_{\rho}S\cap \Delta^-)\}.$$
In Section~\ref{ABD}) we show that the above set contains a two-dimensional lattice cone, so
 $Y_L(0)\cup\{0\}$ contains a two-dimensional lattice cone.

We start from the following useful lemma.

\subsubsection{}
\begin{lem}{}
For all $w\in W^{\#}$ and $\mu\in\mathbb{Z}_{\geq 0}S$
we have
\begin{equation}\label{PSwmu}
|w\mu|=wP_S(|w\mu|).\end{equation}
\end{lem}
\begin{proof}
By~(\ref{wmudef}) we have
$w^{-1}|w\mu|\in\mathbb{Z}S$.
Therefore $P_S(w^{-1}|w\mu|)=w^{-1}|w\mu|$.  Since $w^{-1}|w\mu|-|w\mu|\in Q^{\#}$
we have $P_S(|w\mu|)=P_S(w^{-1}|w\mu|)=w^{-1}|w\mu|$ as required.\end{proof}

\subsubsection{}
\begin{cor}{YJmu0}
For any $\mu_0 \in\mathbb{Z}S$ the set  $Y_{\cJ}(\mu_0)$ lie in a finite union of 
affine lattice cones of dimension one. 
\end{cor}
 \begin{proof}
By~(\ref{Janmu}), if $(k\Lambda_0+\rho-\nu)\in \supp\bigl(Re^{\rho}\sum\limits_{i=1}^\infty \ch \cJ^i(k)\bigr)$, then
 $\nu=r\gamma+\alpha$ for some 
$ \gamma\in {\Delta}^+\setminus \dot{\Delta}$, $r\in\mathbb{Z}_{\geq 1}$ and
$\alpha\in\supp\dot{R}$.  Therefore
$$\dot{P}(\supp(q_{\cJ}))\subset  \{-r\dot{\gamma}-\alpha|\ \ 
 r\in\mathbb{Z}_{\geq 1},\ \dot{\gamma}\in \dot{\Delta}\cup\{0\},\ \ \alpha\in\supp\dot{R}\}.$$
  Combining~(\ref{Janmu}) and~(\ref{denomexp}) we conclude that
 each element in  $\dot{P}(\supp(q_{\cJ}))$ can be written as 
$-(r\dot{\gamma}-w\dot{\rho}+\dot{\rho}+|w\mu|)$ for some
$\dot{\gamma}\in\dot{\Delta}\cup\{0\}$, $w\in\dot{W}^{\#}$ and $\mu\in\mathbb{Z}_{\geq 0}S$. 
Fix $\mu_0 \in\mathbb{Z}S$. Recall that
 $$Y_{\cJ}(\mu_0):=\{\eta\in\dot{Q}_+|\ -(\eta+\mu_0)\in \dot{P}(\supp(q_{\cJ}))\}.$$

Since $\dot{\Delta}$ and $w\in\dot{W}^{\#}$ are finite, it is enough to check that 
for each 
$\dot{\gamma}\in\dot{\Delta}\cup\{0\}$ and $w\in\dot{W}^{\#}$ the set
$$X:=\{\eta\in \dot{Q}_+|\ \ \eta+\mu_0= 
r\dot{\gamma}+|w\mu| \text{ for some } r\in\mathbb{Z}_{\geq 1}, \mu\in  \mathbb{Z}_{\geq 0}S\}$$
lies in a one-dimensional affine lattice cone.  
Take $\eta\in X$. There exists $r\in\mathbb{Z}_{\geq 1}$ and $\mu\in  \mathbb{Z}_{\geq 0}S$ such that 
$$\mu_0=P_S(r\dot{\gamma}+|w\mu| ),\ \ \eta=r\dot{\gamma}+|w\mu| -\mu_0.$$
By~(\ref{PSwmu}) we have $\mu_0=rP_S(\dot{\gamma})-w^{-1}|w\mu|$ so 
$$\eta=r\dot{\gamma}+|w\mu| -\mu_0=r\dot{\gamma}+w\mu_0-r wP_S(\dot{\gamma})-\mu_0=\zeta_0+r\zeta$$
for $\zeta_0:=w\mu_0-\mu_0$ and $\zeta:=\dot{\gamma}- wP_S(\dot{\gamma})$.
We conclude that $X=\zeta_0+\mathbb{Z}_{\geq 1}\zeta$, so $X$ is a  affine lattice cone 
of zero dimension if $\dot{\gamma}=0$ and of dimension one if $\dot{\gamma}\not=0$.
\end{proof}

 \subsubsection{}
 \begin{lem}{lemreg}
For $\mu_0=\sum\limits_{\beta\in S} m_{\beta}\beta\in\mathbb{Z}S$ we have
 $$(k+h^{\vee})\eta\in Y_L(\mu_0)\ \ \Longleftrightarrow\ \ \sum\limits_{w\in C(\eta)} (-1)^{p(\mu)}\sgn (w)\not=0$$
 where $$C(\eta):=\{y\in  \Stab_{W^{\#}} (k\Lambda_0+\rho-\mu_0)|\ t_{-\eta}y \beta\in\Delta^+\ 
 \text{ if and only if }m_{\beta}\geq 0\ \text{ and } t_{-\eta}y \not\in \dot{W} \}.
 $$
  \end{lem}
 \begin{proof}
The coefficient of $e^{-\nu}$ in $q_L$ is equal to 
 $$a_{\nu}:=\sum\limits_{(w,\mu)\in B(\nu)} (-1)^{p(\mu)}\sgn (w)$$
for $B(\nu):=\{(w,\mu)| \ w\in W^{\#}\setminus \dot{W},\ \mu\in\mathbb{Z}_{\geq 0}S,\  
 |w\mu|+(k\Lambda_0+\rho)-w(k\Lambda_0+\rho)=\nu\}$.

For $\nu:=(k\Lambda_0+\rho)-t_{-\eta}(k\Lambda_0+\rho-\mu_0)$ we have 
$$B=\{(w,\mu)| \ w\in W^{\#}\setminus \dot{W},\ \mu\in\mathbb{Z}_{\geq 0}S,\  w(k\Lambda_0+\rho)-|w\mu|=t_{-\eta}
(k\Lambda_0+\rho-\mu_0)\}.$$

One has
$$\dot{P}(\nu)=\dot{P}\bigl((k\Lambda_0+\rho-\mu_0)-t_{-\eta}(k\Lambda_0+\rho-\mu_0)\bigr)+\mu_0=
-(k+h^{\vee})\eta+\mu_0.$$
In particular, $P_S(\nu)=\mu_0$.
Take any  $(w,\mu)\in B$. One has 
$$\mu_0=P_S(\nu)=P_S\bigl((k\Lambda_0+\rho)-w(k\Lambda_0+\rho)+|w\mu|\bigr)=P_S(|w\mu|).$$
Therefore $P_S(|w\mu|)=\mu_0$ and, by~(\ref{PSwmu}), $|w\mu|=w\mu_0$, which implies  
$\mu:=\sum\limits_{\beta\in S} m'_{\beta}\beta\in\mathbb{Z}S$ for $m'_{\beta}:=m_{\beta}$
if $m_{\beta}\geq 0$ and $m'_{\beta}=-m_{\beta}-1$ if $m_{\beta}<0$.  Note that $\mu\in\mathbb{Z}_{\geq 0}S$.
For $\mu$ as above the pair $(w,\mu)$ lies in $B(\nu)$ if and only if $|w\mu|=w\mu_0$
and
$$t_{-\eta}(k\Lambda_0+\rho-\mu_0)=w(k\Lambda_0+\rho)-|w\mu|=w(k\Lambda_0+\rho-\mu_0)$$
that is $t_{\eta}w\in \Stab_{W^{\#}} (k\Lambda_0+\rho-\mu_0)$. The condition $|w\mu|=w\mu_0$ means that
$w\beta\in\Delta^+$ if and only if $m_{\beta}\geq 0$.  Thus
$$a_{\nu}:= (-1)^{p(\mu)} \sum\limits_{y\in C(\eta)}\sgn (y).$$
Since $\dot{P}(\nu)=-(k+h^{\vee})\eta+\mu_0$ we have $(k+h^{\vee})\eta\in Y_L(\mu_0)$ if and only if $a_{\nu}\not=0$.
\end{proof}

 \subsubsection{}
 \begin{cor}{coreta}
Let $\mu_0=\sum\limits_{\beta\in S} m_{\beta}\beta\in\mathbb{Z}S$  and
 $\eta\in Q'\setminus\{0\}$ be such  
 $$(\eta,\beta)\geq 0\ \Longleftrightarrow \ m_{\beta}\geq 0.$$
 Then $\Id\in C(\eta)$. Moreover,   $(k+h^{\vee})\eta\in Y_L(\mu_0)$ if $(k\Lambda_0+\rho-\mu_0)$ is $W^{\#}$-regular.
  \end{cor}

%

\subsubsection{}
 \begin{cor}{correg2}
 If  $\mathfrak{psl}(n|n)$ and  $(k\Lambda_0+\rho-\mu_0)$ is $W^{\#}$-regular
 for some  $\mu_0\in\mathbb{Z}S$, then $Y_L(\mu_0)\cup\{0\}$
 contains a lattice cone of the same rank as $\dot{\fg}^{\#}$.
  \end{cor}
\begin{proof}
Write $\mu_0=\sum\limits_{\beta\in S} m_{\beta}\beta$.
By~\Cor{coreta}, $Y_L(\mu_0)\cup\{0\}$ contains
$(k+h^{\vee})N$, where 
$$N:=\{\nu\in Q'||\ (\nu,\beta)\geq 0 \text{ if and only if } 
m_{\beta}\geq 0\}.$$

The restriction of $(-,-)$ to $\mathbb{Q}\dot{\Delta}^{\#}$  is positively definite.
We write $\alpha\in \mathbb{Q}\dot{\Delta}$ in the form $\alpha=\alpha^{\#}+\alpha^{\perp}$, where 
$\alpha\in \mathbb{Q}\dot{\Delta}^{\#}$ and $(\alpha^{\perp},\dot{\Delta}^{\#})=0$.
It sufficies  to verify that   $\{\beta^{\#}\}_{\beta\in S}$ are linearly independent.

 For $\dot{\fg}$ type I ($\fosp(2|2n)$  and $\fsl(m|n)$ with $m>n$) this  can be easily checked. 
 Let us verify the assertion for  $\dot{\fg}$  of type II. 
Assume that $\{\beta^{\#}\}_{\beta\in S}$ are linearly dependent. Then 
there exists $\nu=\sum\limits_{\beta\in S} a_{\beta}\beta\in\mathbb{Q}S$ such that $\nu\not=0$ and  $\nu^{\#}=0$.
Since $\nu\in\mathbb{Q}S$ we have
$(\nu,\nu)=0$, so   $\nu^{\#}=0$ forces $(\nu^{\perp},\nu^{\perp})=0$. 
Since $\dot{\fg}$ is of type II,  we have $\nu^{\perp}\in\mathbb{Q}\dot{\Delta}'$, where 
$\dot{\Delta}_0=\dot{\Delta}^{\#}\coprod\dot{\Delta}'$ and the  restriction of $(-,-)$ to $\mathbb{Q}\dot{\Delta}^{\#}$  is negatively definite. Therefore
$(\nu^{\perp},\nu^{\perp})=0$ forces
$\nu^{\perp}=0$. Hence $\nu=0$, a contardiction. This completes the proof.
\end{proof}

\subsubsection{}
 \begin{cor}{corbeta2}   If $\Stab_{W^{\#}}(k\Lambda_0+\rho)=\Stab_{\dot{W}^{\#}}\dot{\rho}$), then 
 $Y_L(0)\cup\{0\}$ contains $(k+h^{\vee})\eta$ for any  $\eta\in Q'$ satisfying the property
$$(\eta,\beta')\geq 0 \ \text{ for all } \beta'\in  (\dot{W}_{\rho}S\cap \Delta^+),\ \ (\eta,\beta')\leq 0  \ \text{ for all } 
\beta'\in  (\dot{W}_{\rho}S\cap \Delta^-).$$
\end{cor}
\begin{proof}
For $\mu_0=0$ we have 
$$C(\eta)=\{y\in  \Stab_{W^{\#}} (k\Lambda_0+\rho)|\ t_{-\eta}y \beta\in\Delta^+\ 
 \text{ and } t_{-\eta}y \not\in \dot{W} \}=\{y\in \dot{W}_{\rho}|\ t_{-\eta}y \beta\in\Delta^+\}.
 $$
Using the assumption on $\eta$ we obtain
$C(\eta)=\{y\in \dot{W}_{\rho}|\ y \beta\in\Delta^+\}$. By~(\ref{denomsgny}) 
$$\sum\limits_{y\in C(\eta)} \sgn(y)=1.$$
By~\Lem{lemreg}, $(k+h^{\vee})\eta\in Y_L(0)\cup\{0\}$.
\end{proof}

\subsection{Proof of~\Thm{thm3}  (ii), (iii) }\label{ABD}
In this section we will complete the proof of~\Thm{thm3} (ii), (iii).
Recall that $\dot{\fg}$ is of non-zero defect and $k\in\mathbb{Z}_{>0}$ is such that $k+h^{\vee}\not=0$. 
 We  will choose $\Sigma$ such that
 $(\alpha,\alpha)\geq 0$ for all $\alpha\in\Sigma$, so 
 $(k\Lambda_0+\rho,\alpha)\geq 0$ for $\alpha\in \Sigma(\dot{\fg}_0)$.
 The group
 $\dot{W}_{\rho}=\Stab_{\dot{W}^{\#}}\dot{\rho}$ is 
generated by the reflections $s_{\alpha}$ with $(\dot{\rho},\alpha)=0$.
 
Take $k$ such that $(k\Lambda_0+\rho,\alpha_0)>0$
 (this holds for all $k\in\mathbb{Z}_{\geq 0}$ if
$\alpha_0\in\Sigma$). Since
$\Stab_{W^{\#}}(k\Lambda_0+\rho)$ is generated by  $s_{\alpha}$ with
$\alpha\in \Sigma({\fg}^{\#})$ satisfying $(\rho,\alpha)=0$, we have
$\Stab_{W^{\#}}(k\Lambda_0+\rho)=\dot{W}_{\rho}$.

 \subsubsection{Case $G(3)$, $F(4)$}
We take $\dot{\Sigma}:=\{\vareps_1;
\vareps_2-\delta_1;-\vareps_1+\delta_1\}$ for $G(3)$ and
$$\dot{\Sigma}:=\{\vareps_1-\vareps_2, \frac{1}{2}(\delta_1-\vareps_1+\vareps_2-\vareps_3),
\frac{1}{2}(-\delta_1+\vareps_1+\vareps_2-\vareps_3),\vareps_3\}
$$
for $F(4)$. One has $\alpha_0\in\Sigma$ and  $\dot{W}_{\rho}=\{\Id,s_{\alpha}\}$
where $\alpha:=\vareps_2-\vareps_1$ for $G(3)$ and $\alpha:=\vareps_2-\vareps_3$ for $F(4)$.
Take
 $S=\{\beta\}$ for any isotropic $\beta\in\dot{\Sigma}$.
Then $s_{\alpha}\beta\in\Delta^-$
$$N:=\{\nu\in Q'|\ (\nu,\beta)\geq 0\geq (\nu, s_{\alpha}\beta)\}.$$
  satisfies the assumption of~\Cor{corbeta2}. Note that $N$ contains a lattice cone of the maximal rank in
  $Q'$ which has  rank $2$ for $G(3)$ and rank $3$ for $F(4)$.

\subsubsection{Case $\fsl(n+1|n)$ for $n>1$ and $k>0$}\label{sln+1n}
We fix
$$\dot{\Sigma}:=
\{\vareps_1-\delta_1,\delta_1-\vareps_2,\vareps_2-\vareps_3, \vareps_3-\delta_2,\delta_2-\vareps_4,\vareps_4-\delta_3,
\ldots, \delta_{n-1}-\vareps_{n},\vareps_{n+1}-\delta_n\}$$
and $S=\{\vareps_1-\delta_1\}\cup\{\vareps_{i+1}-\delta_i\}_{i=2}^n$. In this case $\alpha_0\not\in\Sigma$ and 
$(k\Lambda_0+\rho,\alpha_0)>0$ for $k>0$. 
 The group $\dot{W}_{\rho}$ is the  product of $\{\Id, s_{\vareps_1-\vareps_2}\}$ and the  group of permutations of
 $\vareps_3,\ldots,\vareps_{n+1}$.  One has 
 $$\begin{array}{l}
 \dot{W}_{\rho}S\cap \Delta^+=\{\vareps_1-\delta_1\}\cup \{\vareps_i-\delta_j|\  2\leq j\leq n,\  3\leq i\leq j+1\},\\
 \dot{W}_{\rho}S\cap \Delta^-=\{\vareps_2-\delta_1\}\cup \{\vareps_i-\delta_j|\  2\leq j\leq n, \ j+2\leq i\leq n+1\}.\end{array}$$
Note that  for $\beta'\in \dot{W}_{\rho}S$ we have $(\beta',\vareps_i)\in\{0,1\}$ for all $i$; moreover,
 $$\begin{array}{lcr}
 (\beta',\vareps_2)\not=0\ \Longrightarrow\ \beta'\in\Delta^-;\ \ (\beta',\vareps_1)\not=0\ \Longrightarrow  \beta'\in\Delta^+;
\ \  (\beta',\vareps_3)\not=0\ \Longrightarrow  \beta'\in\Delta^+.\end{array}$$

 We set
 $$N:=\{\eta\in Q'|\ (\eta,\vareps_1)\geq 0,\  (\eta,\vareps_2)\leq 0,\  (\eta,\vareps_3)\geq 0, \ (\eta,\vareps_i)=0\ \text{ for }3<i\leq n+1\}.$$
For any $\eta\in N$ we have $(\eta,\beta')\geq 0$ if $\beta'\in (\dot{W}_{\rho}S\cap \Delta^+)$ and
$(\eta,\beta')\leq 0$ if $\beta'\in (\dot{W}_{\rho}S\cap \Delta^-)$, so $\eta$ satisfies the assumption of~\Cor{corbeta2}.
Since 
$$Q'=\{\sum\limits_{i=1}^{n+1} a_i\vareps_i|\ a_i\in\mathbb{Z},\ 
 \sum\limits_{i=1}^{n+1} a_i=0\},$$
  $N$ is a two-dimensional lattice cone.

\subsubsection{Cases $\fsl(m|n)$ with $m-2\geq n>0$, $\fosp(m'|n')$ with $m'-5\geq n'>0$}\label{Am-n2} $ $

For $\fosp(m'|n')$ we set  $n:=n'/2$ and $m:=[m'/2]$: if $m'$ is odd, then $\fosp(m'|n')=B(m|n)$ and $m-2\geq n>0$;
if $m'$ is even, then $\fosp(m'|n')=D(m|n)$ and $m-3\geq n>0$. Set
$$\dot{\Sigma}_A:=\{\vareps_1-\vareps_2,\vareps_2-\delta_1,\delta_1-\vareps_3,\ldots, \delta_n-\vareps_{n+2},
\vareps_{n+2}-\vareps_{n+3},\ldots,\vareps_{m-1}-\vareps_{m}\}.
$$
For $\fsl(m|n)$ we take $\dot{\Sigma}:=\dot{\Sigma}_A$,
for $B(m|n)$ we take $\dot{\Sigma}:=\dot{\Sigma}_A\cup\{\vareps_m\}$;
for $D(m|n)$ we take $\dot{\Sigma}:=\dot{\Sigma}_A\cup\{\vareps_{m-1}+\vareps_{m}\}$.
(For example,  $\dot{\Sigma}=\{\vareps_1-\vareps_2,\vareps_2-\delta_1,\delta_1-\vareps_3,\vareps_3\pm\vareps_4\}$  for $D(4|1)$, and $\dot{\Sigma}=\{\vareps_1-\vareps_2,\vareps_2-\delta_1,\delta_1-\vareps_3,
\vareps_{3}\}$ for $B(3|1)$).
In all cases $\alpha_0\in\Sigma$.
The group $\dot{W}_{\rho}$ is the group of permutations of $\vareps_2,\ldots,\vareps_{n+2}$.
We take $S:=\{\vareps_{i+1}-\delta_i\}_{i=1}^n$. 

Take $\beta'\in \dot{W}_{\rho}S$. Then $\beta'=\vareps_j-\delta_i$ for $1\leq i\leq n$ and $2\leq j\leq n+2$.
Note that $(\beta',\vareps_i)\in\mathbb{Z}_{\geq 0}$ for $i=2,\ldots,n+2$; moreover  $\beta'\in\Delta^+$ if
$(\beta',\vareps_2)\not=0$ and $\beta'\in\Delta^-$ if
$(\beta',\vareps_{n+2})\not=0$.
We set
$$N:=\{\eta\in Q'|\ (\eta,\vareps_2)\geq 0,\ (\eta,\vareps_i)=0\ \text{ for }i=3,\ldots, n+1,\ (\eta,\vareps_{n+2})\leq 0\}.$$
For any $\eta\in N$ we have $(\eta,\beta')\geq 0$ if $\beta'\in \dot{W}_{\rho}S\cap \Delta^+$ and
$(\eta,\beta')\leq 0$ if $\beta'\in \dot{W}_{\rho}S\cap \Delta^-$, so $\eta$ satisfies the assumption of~\Cor{corbeta2}.
Recall that $Q'\cong \mathbb{Z}^{\ell}$ where $\ell=m-1$ for $\fsl(m|n)$ and $\ell=m$ for
$\fosp(m'|n')$.
Note that $N$ is a lattice cone of the maximal rank in $\{\nu\in Q'| \ (\nu,\vareps_i)=0\ \text{ for }i=3,\ldots, n+1 \}\cong \mathbb{Z}^{\ell-(n-1)}$. Since $\ell-(n-1)\geq 2$, $N$ contains a two-dimensional lattice cone.

\subsubsection{$\mathfrak{sop}(n'|m')$ with $n'>m'>0$}\label{BDn-m}
We set $n:=n'/2$ and $m:=[m'/2]$. We fix
$$\dot{\Sigma}=\{\delta_1-\vareps_1,\vareps_1-\delta_2,\ldots, \delta_m-
\vareps_m,\vareps_m-\delta_{m+1},
\delta_{m+1}-\delta_{m+2},\ldots, u\delta_n\}$$
with $u=1$ for $B(m|n)$ and $u=2$ for $D(m|n)$. The group $\dot{W}_{\rho}$ is the group of permutations of $\delta_1,\ldots,\delta_{m+1}$. One has $\alpha_0=\delta-2\delta_1\in\Sigma$.
We take $S=\{\delta_i-\vareps_i\}_{i=1}^m$. 

Take $\beta'\in \dot{W}_{\rho}S$. Then $\beta'=\delta_j-\vareps_i$ for $1\leq i\leq m$ and $1\leq j\leq m+1$.
Note that $(\beta',\delta_i)\in\mathbb{Z}_{\geq 0}$ for $i=1,\ldots,m+1$; moreover  $\beta'\in\Delta^+$ if
$(\beta',\delta_1)\not=0$ and $\beta'\in\Delta^-$ if
$(\beta',\delta_{m+1})\not=0$.
Set
 $$N:=\{\nu\in Q'|\ (\nu,\delta_1)\geq 0,\ (\nu,\delta_i)=0\ \text{ for }i=2,\ldots,m,\ (\nu,\delta_{m+1})\leq 0\}.$$
For any $\eta\in N$ we have $(\eta,\beta')\geq 0$ if $\beta'\in \dot{W}_{\rho}S\cap \Delta^+$ and
$(\eta,\beta')\leq 0$ if $\beta'\in \dot{W}_{\rho}S\cap \Delta^-$, so $\eta$ satisfies the assumption of~\Cor{corbeta2}.
Since $Q'\cong \mathbb{Z}^{n}$, $N$ is a lattice cone of the maximal rank in $\{\nu\in Q'| \ (\nu,\vareps_i)=0\ \text{ for }i=2,\ldots,m\}\cong \mathbb{Z}^{n-m+1}$. Since $n>m$, $N$ contains a two-dimensional lattice cone.

\subsubsection{Case $\fosp(2n+3|2n)$ for $n>0$, $k>0$}\label{Bnn+1}
We fix 
$$\dot{\Sigma}:=
\{\vareps_1-\delta_1,\delta_1-\vareps_2,\ldots, \vareps_n-\delta_n,\delta_n-\vareps_{n+1},
\vareps_{n+1}\},\ \ \ S=\{\vareps_{i}-\delta_i\}_{i=1}^n.$$
In this case $\alpha_0\not\in\Sigma$ and 
$(k\Lambda_0+\rho,\alpha_0)>0$ since $k>0$. 
 The group $\dot{W}_{\rho}$ is the group of permutations of
 $\vareps_1,\ldots,\vareps_{n+1}$.  
 
 Take $\beta'\in \dot{W}_{\rho}S$. Then $\beta'=\vareps_j-\delta_i$ for $1\leq i\leq n$ and $1\leq j\leq n+1$.
Note that $(\beta',\vareps_i)\in\mathbb{Z}_{\geq 0}$ for $i=1,\ldots,n+1$; moreover  $\beta'\in\Delta^+$ if
$(\beta',\vareps_1)\not=0$ and $\beta'\in\Delta^-$ if
$(\beta',\vareps_{n+1})\not=0$. One has 
$$Q'=\{\sum\limits_{i=1}^m a_i\vareps_i|\ a_i\in\mathbb{Z},\ 
 \sum\limits_{i=1}^m a_i\ \text{ is even}\}.$$ 
 Set  $N:=\mathbb{Z}_{\geq 0} 2\vareps_1-\mathbb{Z}_{\geq 0} 2\vareps_{n+1}$.
 For any $\eta\in N$ we have $(\eta,\beta')\geq 0$ if $\beta'\in \dot{W}_{\rho}S\cap \Delta^+$ and
$
(\eta,\beta')\leq 0$ if $\beta'\in \dot{W}_{\rho}S\cap \Delta^-$, so $\eta$ satisfies the assumption of~\Cor{corbeta2}.
 Clearly, $N$ is a two-dimensional lattice cone.

\subsubsection{Case $\fosp(2n+4|2n)$ for $n>0$}\label{Dn+2n}
We fix 
$$
\begin{array}{l}
\dot{\Sigma}:=\{\vareps_1-\vareps_2,\vareps_2-\delta_1,\delta_1-\vareps_3,\ldots, \vareps_{n+1}-\delta_n,\delta_n\pm \vareps_{n+2}\},\ \  S:=\{\delta_i-\vareps_{i+2}\}_{i=1}^n.
\end{array}$$
Then $\alpha_0=\delta-(\vareps_1+\vareps_2)\in\Sigma$. The group $\dot{W}_{\rho}$
is the group of signed permutations of $\vareps_2,\ldots,\vareps_{n+2}$
which change the even number of signs. 
One has 
$$\begin{array}{l}
\dot{W}_{\rho}S\cap\Delta^-=\{\delta_i-\vareps_j|\ 1\leq i\leq n, 2\leq j\leq i+1\}\\
\dot{W}_{\rho}S\cap\Delta^+=\{\delta_i+\vareps_j|\ 1\leq i\leq n, 2\leq j\leq n+2\}\cup\{\delta_i-\vareps_j|\ 1\leq i\leq n, i+2\leq j\leq n+2\}.
\end{array}$$
Take $\beta'\in \dot{W}_{\rho}S$. Then $\beta'=\delta_i\pm \vareps_j$ for $1\leq i\leq n$ and $2\leq j\leq n+2$. 
Note that $(\beta',\vareps_1)=0$.
Moreover, $(\beta',\vareps_2)\in\mathbb{Z}_{\geq 0}$ if $\beta'\in\Delta^+$ and
 $(\beta',\vareps_2)\in\mathbb{Z}_{\leq 0}$ if $\beta'\in\Delta^-$.
Take $\nu:=2a_1\vareps_1+2a_2\vareps_2$ for $a_2\in\mathbb{Z}_{\geq 0}$, $a_1\in\mathbb{Z}$.  Then 
$\nu\in Q'$ and 
$(\nu,\beta')\geq 0$ (resp., $(\nu,\beta')\leq 0$) 
if $\beta'\in \dot{W}_{\rho}S\cap\Delta^+$, (resp., if $\beta'\in \dot{W}_{\rho}S\cap\Delta^-$). 
Thus $\eta$ satisfies the assumption of~\Cor{corbeta2}.
 Clearly, $N=2\mathbb{Z}\vareps_1+2\mathbb{Z}_{\geq 0}\vareps_2$ contains a two-dimensional lattice cone.

\subsection{Proof of~\Thm{thm3}  (iv)}\label{thm3iii}
By~\Cor{correg2}  for $\dot{\fg}\not=\mathfrak{psl}(n|n)$, it is enough to find $\mu_0\in\mathbb{Z}S$ such that 
$$\lambda:=k\Lambda_0+\rho-\mu_0$$ 
is $W^{\#}$-regular. We will do it using the following lemma.

\subsubsection{}
\begin{lem}{lemreg}
 The weight $\lambda=k\Lambda_0+\rho-\mu_0$ is $W^{\#}$-regular if and only if $(\lambda,\alpha)$ is  not divisible by $k+h^{\vee}$
for any $\alpha\in (\dot{\Delta}^{\#})^+$.
\end{lem}
\begin{proof}
By~\cite{Kbook}, Prop. 3.12 a weight $\lambda$  is 
$W^{\#}$-regular if and only if $(\lambda,\alpha)\not=0$ for any 
real root $\alpha\in \Delta^{\#}$. Since any real root in $\Delta^{\#}$
takes the form $s\delta\pm \alpha$ for $\alpha\in (\dot{\Delta}^{\#})^+$, 
$\lambda$ is $W^{\#}$-regular if and only if $(\lambda,\alpha)$ is  not divisible by $(\lambda,\delta)$
for any $\alpha\in (\dot{\Delta}^{\#})^+$. 
\end{proof}

\subsubsection{Case $\mathfrak{spo}(2n|2n+1)=B(n|n)$}
We take 
$${\Sigma}:=\{\delta- \delta_1-\vareps_1; \vareps_1-\delta_1,\delta_1-\vareps_2,\ldots,
\vareps_n-\delta_n, \delta_n\},\ \ \ S=\{\vareps_i-\delta_i\}_{i=1}^n.$$
We have $(\delta_i,\delta_i)=\frac{1}{2}$  and $(\dot{\Delta}^{\#})^+=\{\delta_i\pm\delta_j\}_{1\leq i\leq j\leq n}$. One has
$h^{\vee}=\frac{1}{2}$ and  $(\rho,\delta_i)=\frac{1}{4}$ for $i=1,\ldots,n$.
We set $\mu_0:=\sum\limits_{i=1}^n (n-i) (\vareps_i-\delta_i)$.
Then $(\lambda,\vareps_i)=\frac{2n-2i+1}{4}$ and so for any $\alpha\in (\dot{\Delta}^{\#})^+$
one has 
$$0<(\lambda,\alpha)\leq n-\frac{1}{2}< k+h^{\vee}\ \ \text{ if } k\geq n.$$
 Hence $\lambda$ is $W^{\#}$-regular if $k\geq n$.

\subsubsection{Case $\mathfrak{spo}(2n|2n)=D(n|n)$} 
We take 
$$\dot{\Sigma}:=\{\delta-2\delta_1;\delta_1-\vareps_1,\vareps_1-\delta_2,\ldots, \delta_n-\vareps_{n},
\delta_n+\vareps_n\},\ \ \    S=\{\delta_i-\vareps_i\}_{i=1}^n.$$
We have $(\delta_i,\delta_j)=\frac{1}{2}\delta_{ij}$  and $(\dot{\Delta}^{\#})^+=\{\delta_i\pm\delta_j\}_{1\leq i\leq j\leq n}$,
$h^{\vee}=1$ and $\rho=\Lambda_0$.
We   take
$\mu_0:=-\sum\limits_{i=1}^n (n+1-i)(\vareps_i-\delta_i)$.
Then $(\lambda,\delta_j)=\frac{n+1-i}{2}$ and so for any $\alpha\in (\dot{\Delta}^{\#})^+$
one has 
$$0<(\lambda,\alpha)\leq n+\frac{1}{2}< k+h^{\vee}\ \ \text{ if } k\geq n.$$
Hence $\lambda$ is $W^{\#}$-regular if $k\geq n$.

\subsubsection{Case $\mathfrak{osp}(2n+2|2n)=D(n+1|n)$} 
We take
$${\Sigma}:=\{\delta- \delta_1-\vareps_1; \vareps_1-\delta_1,\delta_1-\vareps_2,\ldots, \vareps_{n}-\delta_n,\delta_n-\vareps_{n+1},
\delta_n+\vareps_{n+1}\},\ \ S:=\{\vareps_i-\delta_i\}_{i=1}^n.$$
Then  $h^{\vee}=0$, $\rho=0$.
We have $(\vareps_i,\vareps_j)=\delta_{ij}$  and 
$(\dot{\Delta}^{\#})^+=\{\vareps_i\pm\vareps_j\}_{1\leq i<j\leq n+1}$.

For $\mu_0:=-\sum\limits_{i=1}^n (n+1-i)(\vareps_i-\delta_i)$ we have
$(\lambda,\vareps_i)=n+1-i$ for $i=1,\ldots,n+1$. For $1\leq i<j\leq n+1$ we have 
$$0<(\lambda,\vareps_i\pm \vareps_j)\leq 2n-1.$$
Hence 
$\lambda$ is  $W^{\#}$-regular if $k+h^{\vee}=k\geq 2n$.

\subsubsection{Case $\mathfrak{psl}(n|n)$}
Consider the remaining case $\dot{\fg}=\mathfrak{psl}(n|n)$ for $n\geq 3$.
One has $h^{\vee}=0$. We  normalize $(.\, ,.\, )$   by $(\vareps_i,\vareps_j)=\delta_{ij}$.  
We take 
$${\Sigma}:=\{\delta-\vareps-1+\delta_n;\vareps_1-\delta_1,\delta_1-\vareps_2,\ldots, \vareps_{n}-\delta_n\},\  \ \  
S:=\{\vareps_i-\delta_i\}_{i=1}^n.$$
Then $\rho=0$. Take $\mu_0:=-\sum_{i=1}^n i (\vareps_i-\delta_i)$.
Then  for $i<j$ we have
$1-n\leq (\lambda,\vareps_i-\vareps_j)<0$, so 
$\lambda$ is  $W^{\#}$-regular if $k+h^{\vee}=k\geq n$.
In the light of~\Cor{coreta}, $Y_L(\mu_0)$ contains
$(k+h^{\vee})N$, where 
$$N:=\{\sum a_i\vareps_i|\ a_i\in\mathbb{Z},\ \sum\limits_{i=1}^n a_i=0,\ \ a_1>0,\  a_i<0\ \text{ for } i=2,\ldots,n\}.$$
Observe that $N\cup\{0\}$ contains the   lattice cone 
 $\sum\limits_{i=2}^n \mathbb{Z}_{\geq 1}(\vareps_1-\vareps_i)$ of dimension $n-1$.
\qed

\subsection{Proof of~\Thm{thm3} (v)}\label{spo2n2n+2}
We consider   $\dot{\fg}:=\mathfrak{spo}(2n|2n+2)$ and fix
$$\Sigma:=\{\delta-\vareps_1-\delta_1; \vareps_1-\delta_1,\delta_1-\vareps_2,\ldots,
\vareps_n-\delta_n,\delta_n-\vareps_{n+1}\}, \ \ S:=\{\vareps_i-\delta_i\}.$$
One has $\rho=0$.

We have $\mathbb{Q}\dot{\Delta}=E_+\oplus E_-$, where $E_+$ is the span
of $\delta_1,\ldots,\delta_n$ and $E_-$ is the span of $\vareps_1,\ldots,\vareps_{n+1}$.
The restriction of $(.\, ,.\,)$ to $E_+$ (resp., to $E_-$) is positively
(resp., negatively) definite (we have $(\delta_i,\delta_j)=\frac{1}{2}
\delta_{ij}=-(\vareps_i,\vareps_j)$).
For $\nu\in \mathbb{Q}\Delta$ we write $\dot{\nu}=\nu_-+\nu_+$ with
$\nu_{\pm}\in E_{\pm}$ and 
denote by $P_{\pm}$ the projections $P_{\pm}: \mathbb{Q}\Delta\to E_{\pm}$
given by $P_{\pm}(\nu):=\nu_{\pm}$.

In this case $\dot{\fg}_{\ol{0}}=\mathfrak{o}_{2n+1}\times \mathfrak{sp}_{2n}$
and $\dot{\Delta}^{\#}\subset E_+$ is the root system of type $C_n$.
Using the notation of Section~\ref{sectdenomfin} we have $\dot{W}''=\dot{W}^{\#}$
and $\dot{W}'$ is the Weyl group of type $D_{n+1}$.
Recall that $W':=T\times \dot{W}'$ and $\rho=0$, so
$$q_L=e^{-k\Lambda_0}\sum\limits_{\nu\in Q'\setminus\{0\}} 
t_{\nu}\bigl(\displaystyle\cF_{\dot{W}'}
(\frac{e^{k\Lambda_0}}{\prod\limits_{\beta\in S}(1+e^{-\beta})})\bigr).
$$

\subsubsection{}
\begin{lem}{lemqL2n+22n}
For $k>2n$  the set
$\supp (q_L)$ contains $-(t_{\eta}\mu_0+k\eta)$ for $\mu_0:=\sum\limits_{i=1}^n i(\vareps_i-\delta_i)$ and 
any $\eta\in Q'$ with $(\eta,\delta_i)>0$ for all $i=1,\ldots, n$.
\end{lem}
\begin{proof}
The coefficient of $e^{-\nu}$ in $q_L$ is equal to
$$a_{\nu}:=\sum\limits_{(w,\mu)\in B(\nu)} (-1)^{p(\mu)}\sgn (w)$$
for $B(\nu):=\{(w,\mu)| \ w\in W'\setminus \dot{W}',\ \mu\in\mathbb{Z}_{\geq 0}S,\  
 |w\mu|+k\Lambda_0-w(k\Lambda_0)=\nu\}$.

Take  $\eta\in Q'$ such that $(\eta,\delta_i)>0$ for all $i=1,\ldots, n$ and set 
$\nu:=|t_{\eta}\mu_0|+k\Lambda_0-w(k\Lambda_0)$.  One has 
$t_{\eta}S\subset\Delta^+$, so 
$|t_{\eta}\mu_0|=t_{\eta}\mu_0$. This gives
$$\nu=t_{\eta}\mu_0+k\eta,\ \ (t_{\eta},\mu_0)\in B(\nu).$$
It is enough to verify that 
\begin{equation}\label{BnuD}
B_{\nu}=\{(t_{\eta},\mu_0)\}.
\end{equation}
Take $(w,\mu)\in B(\nu)$. Recall that 
$|w\mu|=w\mu'$ for some $\mu'\in\mathbb{Z}S$.
Since
 $w\in W'=T\times \dot{W}'$ we have $w=t_{\eta'}y$ with $y\in  \dot{W}'$ and $\eta'\in Q'$. Then 
 $(w,\mu)\in B(\nu)$ gives
$$\nu=|w\mu|+k\Lambda_0-w(k\Lambda_0)=t_{\eta'}y\mu'+k\eta'.$$
Since $\mu'\in\mathbb{Z}S$ we have $\mu'=:\sum\limits_{i=1}^n m'_i(\vareps_i-\delta_i)$. Therefore
$$\begin{array}{l}
\sum\limits_{i=1}^n i\vareps_i =P_-(\nu)=P_-(t_{\eta'}y\mu'+k\eta')=P_-(y\mu')=
\sum\limits_{i=1}^n m'_i (y \vareps_i),\\
-\sum\limits_{i=1}^n i\delta_i +k\eta=P_+(\mu_0+k\eta)=P_+(t_{\eta'}y\mu'+k\eta')=P_+(y\mu')+k\eta'=
-\sum\limits_{i=1}^n m'_i \delta_i
+k\eta'.\end{array}$$
The first formula gives 
$\{|m'_i|\}_{i=1}^n=\{1,2,\ldots,n\}$ and the second formula implies that
$i-m'_i$ is divisible by $k$ for all $i=1,\ldots,n$.  Since $k>2n$ this forces
$m'_i=i$ for all $i=1,\ldots,n$ that is $\mu'=\mu_0$. Then
 the second formula gives $\eta'=\eta$. Since $\mu'=\mu_0\in\mathbb{Z}_{\geq 0}S$ we have
 $wS\subset\Delta^+$  and $\mu'=\mu$. Hence
 $\mu=\mu_0$ and $w\mu_0=t_{\eta} \mu_0$.
 By above, $w=t_{\eta}y$ so $y\mu_0=\mu_0$. Recall that $y$ is a signed permutations
 of $\{\vareps_i\}_{i=1}^{n+1}$ which changes an even number of signs.
 Since $\mu_0=\sum\limits_{i=1}^n i(\vareps_i-\delta_i)$, this gives $y=\Id$. This establishes~(\ref{BnuD}). 
\end{proof}

\subsubsection{}
\begin{lem}{lemqJ2n+22n}
For each $\zeta\in E_-$ the set $$\{\nu\in E_+|\ -(\nu+\zeta)\in \dot{P}(\supp(q_{\cJ})\}$$
is a finite union of one-dimensional affine lattice cones.
\end{lem}
\begin{proof}
Let $\phi: E_-\to E_+$ be the map given by
$\phi(\sum\limits_{i=1}^{n+1} m_i\vareps_i):=\sum\limits_{i=1}^{n} m_i\delta_i$.
For any $\mu\in\mathbb{Z}S$ we have
$$\mu=P_+(\mu)+P_- (\mu),\ \  P_+(\mu)=-\phi(P_+(\mu)).$$

Take
$\nu\in E_+$ such that $-(\nu+\zeta)\in \dot{P}(\supp(q_{\cJ}))$. By~(\ref{Janmu})
 there exist $r\geq 1$ and $\dot{\gamma}\in \dot{\Delta}\cup\{0\}$ 
such that $\dot{K}(\nu+\zeta+r\dot{\gamma})\not=0$. The denominator identity gives
$$\nu+\zeta+r\dot{\gamma}=y\mu\ \ \text{ for some $y\in\dot{W}'$ and }\ \mu\in\mathbb{Z}S.$$
Therefore
$$\zeta+r P_-(\dot{\gamma})=P_-(y\mu)=y P_-(\mu),\ \ \ \nu+rP_+(\dot{\gamma})=P_+(y\mu)=P_+(\mu).$$
Then
$$\nu=P_+(\mu)-rP_+(\dot{\gamma})=-\phi(P_-(\mu))-rP_+(\dot{\gamma})=-\phi\bigl(y^{-1}\zeta+r y^{-1}P_-(\dot{\gamma})\bigr)-rP_+(\dot{\gamma})$$
so $\nu=-\phi\bigl(y^{-1}\zeta)-r\bigl( \phi(y^{-1}P_-(\dot{\gamma}))+P_+(\dot{\gamma})\bigr)$.
We see that for a fixed pair $(\dot{\gamma}, y)$ the suitable values of $\nu$ lie in a one-dimensional affine
lattice cone. Since $\dot{\gamma}\in \in \dot{\Delta}\cup\{0\}$  and $y\in\dot{W}'$,
the numbwer of pairs is finite. The assertion follows.
\end{proof}

\subsubsection{}
By~\Lem{lemqL2n+22n} for $\zeta:=P_-(\mu_0)$ the set 
$\{\nu\in E_+|\ -(\nu+\zeta)\in \dot{P}(\supp(q_{L})\}$ contains an affine lattice cone of dimension
$n>1$. Using~\Lem{lemqJ2n+22n} we obtain $q_{\cJ}+q_L\not=0$ as required.

\section{Appendix}
\label{sectsl2}
In this Appendix we prove a lemma on Verma modules over the affine Lie algebra
$\ft:=\hat{\fsl}_2$ with the simple roots $\alpha_1=\alpha$, $\alpha_0:=\delta-\alpha$
and the invariant bilinear form normalized by $(\alpha,\alpha)=2$.  
Its corollary is used in the proof of~\Thm{thm2}. 
We denote by $w\circ$ the usual shifted action of the Weyl group $W(\ft)$ on $\fh^*$:
$w\circ\lambda:=w(\lambda+\rho_{\ft})-\rho_{\ft}$.

Let  $\lambda\in\fh^*$ and let $N_{\ft}(\lambda)$ be
the maximal locally finite $\fsl_2$-quotient of $M_{\ft}(\lambda)$, and let $N'_{\ft}(\lambda)$
be the maximal proper submodule of $N_{\ft}(\lambda)$. One has
$$N_{\ft}(\lambda)\not=0\ \ \ \Longrightarrow\ \ \ (\lambda,\alpha)\in\mathbb{Z}_{\geq 0},\ \ 
L_{\ft}(\lambda)=N_{\ft}(\lambda)/N'_{\ft}(\lambda).$$

We will use the following lemma which computes
 $N'_{\ft}(\lambda)$ in the case $(\lambda,\delta)\in\mathbb{Z}_{\geq -1}$
(a similar approach works in the case when  $(\lambda,\delta)\not=-2$).

\subsection{}
\begin{lem}{lemsl2}
Let   $\lambda\in\fh^*$ be such that $k:=(\lambda,\delta)\in\mathbb{Z}_{\geq -1}$ and $(\lambda,\alpha)\in\mathbb{Z}_{\geq 0}$. We set
\begin{equation}\label{jrdef}
j:=\lceil \frac{(\lambda,\alpha)+1}{k+2}\rceil,\ \ \  r:=j(k+2)- (\lambda,\alpha)-1
\end{equation}
(where  $\lceil x\rceil$  stands for the minimal integer $\geq x$). Then
$N'_{\ft}(\lambda)=0$ if  $r=0$, and 
$$N'_{\ft}(\lambda)=L_{\ft}(\lambda_1)\ \ \text{ for }\ \ \lambda_1=s_{j\delta-\alpha}\circ\lambda=\lambda-r(j\delta-\alpha),\ \ \text{
if }r>0.$$
\end{lem}
\begin{proof}
Since $k\in\mathbb{Z}_{\geq -1}$,
the $W(\ft)\circ$-orbit of $\lambda$ contains a unique maximal weight $\lambda'$.
Let $W_0:=\{y\in W(\ft)|\ y\circ \lambda'=\lambda'\}$;
each left coset in $W(\ft)/W_0$ contains a unique longest element; we denote the set of these elements
by $Y$ (one has $Y= W(\ft)$ if $W_0=\{\Id\}$). 

The simple subquotients of $M_{\ft}(\lambda')$ are
 of the form $L(y\circ\lambda')$, $y\in Y$, and
$$M_{\ft} (y_2\circ\lambda)\subset M_{\ft}(y_1\circ \lambda)\ \ \Longleftrightarrow\ \ \ 
[M_{\ft}(y_1\circ\lambda):L_{\ft}(y_2\circ\lambda)]\not=0\ \ \Longleftrightarrow\ \ \ y_2\geq y_1$$
where $\geq$ is given by the Bruhat order in the Coxeter group $W(\ft)$.
Moreover, since all Kazhdan-Lusztig polynomials for $\hat{\fsl}_2$ are ``trivial'' we have
$$[M_{\ft}(y_1\circ\lambda):L_{\ft}(y_2\circ\lambda)]=1 \ \text{ if }\ y_2\geq y_1.$$

The Weyl group $W(\ft)$ is  generated by the reflections $s_i:=s_{\alpha_i}$ for $i=0,1$.
By above, $\lambda=w\circ\lambda'$ for some $w\in Y$ such that $s_1w>w$ and 
$$N_{\ft}(\lambda)=M_{\ft}(\lambda)/M_{\ft}(s_1\circ\lambda),\ \ 
L_{\ft}(\lambda)=N_{\ft}(\lambda)/N'_{\ft}(\lambda).$$ 
Therefore all simple subquotients of $N'_{\ft}(\lambda)$ have multiplicity one and
 $$[N'_{\ft}(\lambda):L_{\ft}(\lambda'')]\not=0\ \ \Longleftrightarrow \ \ \lambda''=y\circ \lambda\ \text{ for }\ 
y\in Y_w:=\{y\in Y|\ y>w,\ y\not\geq s_1w\}.$$

Since $W(\ft)$ is the infinite dihedral group  generated by the reflections $s_0,s_1$, 
one has
$$y>y'\ \ \Longleftrightarrow\ \ \ell(y)>\ell(y'),$$
 where $\ell(y)$ stands for the length of $y$.
Since $s_1w>w$ we have $\ell(s_1w)=\ell(w)+1$  which gives
\begin{equation}\label{Yw}
Y_w=\{y\in Y|\ \ \ell(y)=\ell(w)+1,\ y\not=s_1w\}.\end{equation}

Since $\lambda'$ is maximal in $W\circ \lambda'$ we have
$$(\lambda'+\rho_{\ft},\alpha_i)\geq 0\ \text{ for } i=0,1.$$

Consider the case when $W_0\not=\{\Id\}$. 
By~\cite{Kbook}, Prop. 3.12 the subgroup $W_0$ is generated by simple reflections, so  
 $W_0=\{s_i,\Id\}$ for $i=0$ or $i=1$. This gives  $Y=\{(s_0s_1)^ms_0\}_{m\geq 0}$ for $i=0$ and 
$Y=\{s_1(s_0s_1)^m\}_{m\geq 0}$ for $i=1$. Note that $\ell(y)$ is odd for all $y\in Y$, so~(\ref{Yw}) implies
$Y_w=\emptyset$. Hence $N'_{\ft}(\lambda)=0$.  
Write $w\alpha_i$ as $w\alpha_i=\pm (n\delta\pm\alpha)$ 
for $n\geq 0$. Since $s_i\circ\lambda'=\lambda'$ and
$\lambda=w\circ\lambda'$ we have
$$0=(\lambda'+\rho_{\ft},\alpha_i)=(\lambda+\rho_{\ft},w\alpha_i)
=(\lambda+\rho_{\ft},n\delta\pm\alpha)=n(k+2)\pm ((\lambda,\alpha)+1).$$
Both $n$ and $(\lambda,\alpha)$ are non-negative, so  $w\alpha_i=\pm (n\delta-\alpha)$ .
This implies $n=j$, $r=0$, where 
 $j,r$ are as in~(\ref{jrdef}). Summarizing we obtain
 $$W_0\not=\{\Id\}\ \ \ \Longrightarrow\ \ \  N'_{\ft}(\lambda)=0\ \text{ and } r=0.$$

Now consider the remaining case $W_0=\{\Id\}$. In this case
\begin{equation}\label{lambdasl2}
(\lambda'+\rho_{\ft},\alpha_i)>0\ \text{ for } i=0,1.
\end{equation}
Since   $W_0=\{\Id\}$ we have
$Y=W(\ft)$ so $Y_w=\{y\in W(\ft)|\ \ell(y)=\ell(w)+1,\ y\not=s_1w\}$.

If $w=(s_0s_1)^m$,  then $Y_w=\{ws_0\}$  and $ws_0=s_{w\alpha_0}w$ with $w\alpha_0=(2m+1)\delta-\alpha$.

If $w=(s_0s_1)^ms_0$,  then $Y_w=\{ws_1\}$ and $ws_1=s_{w\alpha}w$ with $w\alpha=(2m+2)\delta-\alpha$.

Setting $j':=\ell(w)+1$  and $\gamma:=j'\delta-\alpha$ we get
$N'_{\ft}(\lambda)\cong L(s_{\gamma}\circ\lambda)$ and $w^{-1}\gamma\in \{\alpha,\alpha_0\}$.

It remains to verify that $j'=j$, $s_{\gamma}\circ\lambda=\lambda-r\gamma$ and $r>0$
where $r,j$ are given by~\Lem{lemsl2}.  For  $n\in\mathbb{Z}$ we have
\begin{equation}\label{ira}
(\lambda'+\rho_{\ft},w^{-1}(n\delta-\alpha))=(\lambda+\rho_{\ft},n\delta-\alpha)=(k+2)n-(\lambda,\alpha)-1.
\end{equation}
Thus, by definition,  $j$ is the minimal integer satisfying  $(\lambda'+\rho_{\ft},w^{-1}(j\delta-\alpha))\geq 0$.
By~(\ref{lambdasl2})  we have  $(\lambda'+\rho_{\ft},\beta)>0$  for all  $\beta\in\Delta^+(\ft)$. Therefore
 $j$ is the minimal integer such that $w^{-1}(j\delta-\alpha)\in\Delta^+(\ft)$.
One has
$$w^{-1}(n\delta-\alpha)=w^{-1}(\gamma-(j'-n)\delta)=w^{-1}\gamma-(j'-n)\delta.$$
Since $w^{-1}\gamma\in \{\alpha,\alpha_0\}$, we conclude that $w^{-1}(n\delta-\alpha)\in\Delta^+(\ft)$
is equivalent to $n\geq j'$. Hence $j=j'$.
 By~(\ref{ira}), $s_{\gamma}\circ\lambda=\lambda-r\gamma$. 
Since $N'_{\ft}(\lambda)$ is a  proper submodule of $N_{\ft}(\lambda)$ we have $r>0$.
This completes the proof.
\end{proof}

\subsubsection{}
Set $\Lambda_1:=\Lambda_0+\frac{\alpha}{2}$ (note that $\Lambda_0,\Lambda_1$ are the fundamental weights
for $\alpha_0,\alpha_1)$.

\begin{cor}{corsl2need}
Under the assumptions of~\Lem{lemsl2} we have $N'_{\ft}=L_{\ft}(\lambda_1)$ where 
$\lambda=\lambda_1$ if $(\lambda,\alpha)\leq k+2$ and 
$(\lambda-\lambda_1,\Lambda_1)<(\lambda,\alpha)$ if $(\lambda,\alpha)>k+2$.
\end{cor}
\begin{proof}
One has $j-1<\frac{(\lambda,\alpha)+1}{k+2}$ and $0\leq r\leq k+1$.
If $(\lambda,\alpha)\leq k+2$, then  $r(j-1)=0$.

If $(\lambda,\alpha)>k+2$, then
$$(\lambda-\lambda_1,\Lambda_1)=r(j\delta-\alpha, \Lambda_0+\frac{\alpha}{2})=r(j-1)<(k+1)\frac{(\lambda,\alpha)+1}{k+2}
=(\lambda,\alpha)+1-\frac{(\lambda,\alpha)+1}{k+2}.$$
Since  $(\lambda,\alpha)>k+2$ we have $\frac{(\lambda,\alpha)+1}{k+2}>1$, so  $(\lambda-\lambda_1,\Lambda_1)<(\lambda,\alpha)$.
\end{proof}


\end{document}